\documentclass{aims}

\pdfoutput=0
\usepackage{amsmath}
\usepackage{paralist}
\usepackage{graphics} %% add this and next lines if pictures should be in esp format
\usepackage{epsfig} %For pictures: screened artwork should be set up with an 85 or 100 line screen
\usepackage{graphicx}
\def\currentvolume{X}
 \def\currentissue{X}
  \def\currentyear{200X}
   \def\currentmonth{XX}
    \def\ppages{X--XX}
     \def\DOI{10.3934/xx.xx.xx.xx}
\newcommand{\R}{\mathbb{R}}
\newcommand{\N}{\mathbb{N}}
\newcommand{\Z}{\mathbb{Z}}
\newcommand{\T}{\mathbb{T}}
\newcommand{\RRn}{{\mathbb{R}^n}}
\newcommand{\RRd}{{\mathbb{R}^d}}
\newcommand{\OT}{{[0,T]}}

\renewcommand{\div}{\mathrm{div}}
\newcommand{\grad}{\nabla}

\newcommand{\TT}{\mathbb{T}}
\newcommand{\TTd}{{\TT^d}}

\DeclareMathOperator{\EE}{\mathbb{E}}
\DeclareMathOperator{\PP}{\mathbb{P}}

\newcommand{\cF}{{\mathcal F}}

\newcommand{\cA}{{\mathcal A}}

\newcommand{\cC}{{\mathcal C}}
\newcommand{\cH}{{\mathcal H}}

\newcommand{\cP}{{\mathcal P}}

\newcommand{\cK}{{\mathcal K}}

\newcommand{\cQ}{{\mathcal Q}}

\newcommand{\ds}{\displaystyle}

\newcommand{\beq}{\begin{equation}}
\newcommand{\beqa}{\begin{eqnarray}}
\newcommand{\bea} {\begin{array}{ll}}
\newcommand{\beqan}{\begin{eqnarray*}}
\newcommand{\eeq}{\end{equation}}
\newcommand{\eeqa}{\end{eqnarray}}
\newcommand{\eeqan}{\end{eqnarray*}}
\newcommand{\eea} {\end{array}}
\newtheorem{theorem}{Theorem}
\newtheorem{corollary}{Corollary}

\newtheorem{lemma}[theorem]{Lemma}
\newtheorem{proposition}{Proposition}

\theoremstyle{definition}

\newtheorem{remark}{Remark}

\numberwithin{equation}{section}
\title[Mean field type control]{On the system of partial differential equations arising in mean field type control}

\author[Yves Achdou and  Mathieu Lauri{\`e}re ]{}



\email{achdou@ljll.univ-paris-diderot.fr}
\email{mathieu.lauriere@gmail.com}
\begin{document}
\begin{abstract}
We discuss the system of Fokker-Planck and Hamilton-Jacobi-Bellman equations arising
from the finite horizon control of McKean-Vlasov dynamics. We give examples of
 existence and uniqueness results. Finally, we propose some simple models for the motion of pedestrians and
 report about  numerical simulations in which  we compare
 mean filed games and mean field type control.
\end{abstract}

\maketitle
\centerline{\scshape Yves Achdou }
\medskip
{\footnotesize
\centerline{Universit{\'e} Paris Diderot}
\centerline{Laboratoire Jacques-Louis Lions, UMR 7598, UPMC, CNRS}
\centerline{Sorbonne Paris Cit{\'e} F-75205 Paris, France}
}
\medskip
\centerline{\scshape Mathieu Lauri{\`e}re }
\medskip
{\footnotesize
\centerline{Universit{\'e} Paris Diderot}
\centerline{Laboratoire Jacques-Louis Lions, UMR 7598, UPMC, CNRS}
\centerline{Sorbonne Paris Cit{\'e} F-75205 Paris, France}
}

\bigskip
\centerline{(Communicated by the associate editor name)}

%%%%%%%%%%%%%%%%%%%%%%%%%%%%%%%%%%%%%%%%%%%%
\section{Introduction}
In the recent years,  an important research activity has been devoted to the study of stochastic differential games  with a large number of players.
In their pioneering articles  \cite{MR2269875,MR2271747,MR2295621}, J-M. Lasry and P-L. Lions have introduced the notion of mean field games,
which describe the   asymptotic behavior of  stochastic differential games (Nash equilibria) as the number $N$ of players
tends to infinity.  In  these models, it is assumed that the agents are all identical and that
an individual agent can hardly influence the outcome of the game.  Moreover, each individual strategy is influenced by some averages of functions of
the states of the other agents. In the limit when $N\to +\infty$, a given agent feels the presence of the other agents through the
statistical distribution of the states of the other players. Since perturbations of a single agent's strategy does not influence the statistical distribution of the states,
the latter acts as a parameter  in the control problem to be solved by each  agent.
\\
Another kind of asymptotic regime is obtained by  assuming that all the agents use the same distributed feedback strategy
 and by passing  to the limit as $N\to \infty$ before optimizing the common feedback. Given a common feedback strategy, the asymptotics are
given by the McKean-Vlasov theory, \cite{MR0221595,MR1108185} : the dynamics of a given agent is found by solving  a stochastic differential equation
with coefficients depending on
a mean field, namely the statistical distribution of the states, which may also affect the objective function. Since the feedback strategy is common to all agents,
perturbations of the latter affect the mean field.  Then, having each player optimize its objective function amounts to solving a control problem
 driven by the McKean-Vlasov dynamics. The latter is named control of McKean-Vlasov dynamics by R. Carmona and F. Delarue \cite{MR3045029,MR3091726} and mean field type control by A. Bensoussan et al, \cite{MR3037035,MR3134900}.
\\
When the dynamics of the players are independent stochastic processes, both mean field games and control of  McKean-Vlasov dynamics
naturally lead to a coupled system of partial differential equations, a forward Fokker-Planck equation  (which may be named  FP equation in the sequel)
 and a backward Hamilton-Jacobi--Bellman equation (which may be named  HJB equation).
For mean field games, the coupled system of partial differential equations has been studied by Lasry and Lions in  \cite{MR2269875,MR2271747,MR2295621}. Besides, many important aspects of the mathematical theory developed by  J-M. Lasry and P-L. Lions on MFG are not published in journals or books, but can be found in the videos of  the lectures of P-L. Lions at Coll{\`e}ge de France: see the web site of Coll{\`e}ge de France, \cite{PLL}. One can also see \cite{MR3195844} for a brief survey. \\
In the present paper, we aim at studying the system of partial differential equations arising in mean field type control, when the horizon of the control problem is finite: we will discuss the existence and the uniqueness of classical solutions. In the last paragraph of the paper,
we briefly discuss some numerical simulations in the context of motion of pedestrians, and we compare the results obtained with  mean field games and with mean field type control.
\subsection{Model and assumptions}
For simplicity, we assume that all the functions used below (except in \S~\ref{sec:some-simulations})
 are periodic with respect to the state variables $x_i$, $i=1,\dots, d$,
of period $1$ for example. This will save technical arguments on either problems in unbounded domains or boundary conditions.
 We denote by $\T^d$ the $d-$dimensional unit torus: $\T^d=\R^d/\Z^d$.
Let $\PP$ be the set of probability measures on $\T^d$ and  $\PP\cap L^1(\T^d)$ be the set of probability measures which are absolutely continuous
 with respect to the Lebesgue measure. For $m\in \PP\cap L^1(\T^d)$, the density of $m$ with respect to the Lebesgue measure
will be still be noted $m$,
i.e. $dm(x)=m(x)dx$.
\\
Let $g$ be a map from $\PP$ to a subset of  $\cC^1(\T^d \times \R^n; \R^d)$
( the image of $m\in \PP$ will be noted $g[m]\in   \cC^1(\T^d \times \R^n; \R^d)\;$)
such that
\begin{itemize}
\item  there exists a constant $M$ such that
  for all $m\in \PP$  and $x\in \T^d$,   $|g[m](x,0)| \le M$
\item there exists a constant $L$ such  that
  \begin{itemize}
  \item for all $m\in \PP$, $ a\in \R^n$ and $x,y\in \T^d$,   $|g[m](x,a)-g[m](y,a)| \le Ld(x,y)$
where $d(x,y)$ is the distance between $x$ and $y$ in $\T^d$.
  \item for all $m\in \PP$, $ a,b\in \R^n$ and $x\in \T^d$, $|g[m](x,a)-g[m](x,b)| \le L|a-b|$
  \item for all $m, m' \in \PP$, $ a\in \R^n$ and $x\in \T^d$, $|g[m](x,a)-g[m'](x,a)| \le L d_2(m,m')$ where $d_2$ is the Wasserstein distance:
    \begin{displaymath}
      \begin{split}
        & d_2(m,m')\equiv \inf_{ \gamma\in \Gamma(m,m')}  \left( \int_{\T^d\times \T^d} d^2(x,y) d\gamma(x,y)\right)^{\frac 1 2},        \\
        & \Gamma(m,m')\equiv \Bigl\{  \gamma :\hbox{ transport plan between $m$ and  $m'$}\Bigr\},
      \end{split}
    \end{displaymath}
    and a transport plan $\gamma$ between $m$ and $m'$ is a   Borel probability measure on $\T^d\times \T^d$ such that, for all  Borel subset $E$ of $\T ^d$,
    \begin{displaymath}
      \gamma(E\times \T^d)=m(E)\quad \hbox{and}\quad \gamma(\T^d \times E)=m'(E).
    \end{displaymath}
  \end{itemize}
  \item there exists a map $\tilde g$ from $L^1 (\T^d)$ to  $\cC^1(\T^d \times \R^n; \R^d)$
    such that  $g|_{\PP\cap L^1(\T^d)}= \tilde g|_{\PP\cap L^1(\T^d)}$ and that  for any $x\in \T^d$ and $ a\in \R^n$,
    $m\to \tilde g[m](x,a)$ is Fr{\'e}chet differentiable in $L^1 (\T^d)$ and $(x,a)\mapsto \frac {\partial \tilde g}{\partial m} [m](x,a)$
belongs to \\ $\cC^1(\T^d\times\R^n;  L^\infty(\T^d; \R^d))$.
    Hereafter, we will not make the distinction between $g$ and $\tilde g$.
\end{itemize}
Consider a probability space $(\Omega, \cA, \cP)$ and a filtration $\cF^t$ generated by a $d$-dimensional standard Wiener process $(W_t)$ and
the stochastic process $(X_t)_{t\in\OT}$   in $\R^d$ adapted to $\cF^t$ which solves the stochastic differential equation
\begin{equation}\label{eq:1}
  d X_t = g[m_t]( X_t, a_t) \;d t + \sqrt{2\nu} \;d W_t \qquad \forall t\in\OT,
\end{equation}
given the initial state $X_0$ which is a random variable $\cF^0$-measurable whose probability density is noted $m_0$.
In (\ref{eq:1}),  $\nu$ is a positive number, $m_t$ is the probability distribution of $X_t$ and $a_t$ is the control which we take to be
\begin{equation}\label{eq:2}
a_t=v(t,X_t), \end{equation}
where $v(t,\cdot)$ is a continuous function on $\T^d$. To the pair $(v, m)$, we associate the objective
\begin{equation}
  \label{eq:3}
  \begin{split}
    J(v, m) := &\EE\left[ \int_0^T f[m_t]( X_t, a_t) dt  + h[m_T](X_T)\right]
  \end{split}
\end{equation}
where
$f$ (resp. $h$) is a  map from $\PP$ to a subset of  $\cC^1(\T^d \times \R^n)$, resp. to a  subset of  $\cC^1(\T^d)$.
We assume that
\begin{itemize}
\item $\lim_{|a|\to \infty} \inf_{m\in \PP, x\in \T^d} \frac {f[m](x,a)}{|a|}=+\infty$
\item there exists a map $\tilde f$ from $L^1 (\T^d)$ to  $\cC^1(\T^d \times \R^n)$
    such that  $f|_{\PP\cap L^1(\T^d)}= \tilde f|_{\PP\cap L^1(\T^d)}$ and that  for any $x\in \T^d$ and $ a\in \R^n$,
    $m\to \tilde f[m](x,a)$ is Fr{\'e}chet differentiable in $L^1 (\T^d)$ and $(x,a)\mapsto \frac {\partial \tilde f}{\partial m} [m](x,a)$ belongs to
 $\cC^1(\T^d\times\R^n;  L^\infty(\T^d))$.
   Hereafter, we will not make the distinction between $f$ and $\tilde f$.
\end{itemize}
We also assume that  there exists a map $\tilde h$ from $L^1 (\T^d)$ to  $\cC^1(\T^d)$
    such that  $h|_{\PP\cap L^1(\T^d)}= \tilde h|_{\PP\cap L^1(\T^d)}$ and that  for any $x\in \T^d$,
    $m\to \tilde h[m](x)$ is Fr{\'e}chet differentiable in $L^1 (\T^d)$ and $x\mapsto \frac {\partial \tilde h}{\partial m} [m](x)$
 belongs to $\cC^1(\T^d;  L^\infty(\T^d))$.
   Hereafter, we will not make the distinction between $h$ and $\tilde h$.
\\
It will be useful to define the Lagrangian and Hamiltonian as follows:
for any $x\in \T^d$, $a\in \R^n$ and $p\in \R ^d$,
\begin{align*}
	 L[m](x,a,p) &:= f[m]( x,a)+ p\cdot g[m](x,a)\\
	 H[m](x,p) &:= \min_{a\in\RRn} L[m](x,a,p).
\end{align*}
where $ p \cdot q$ denotes the scalar product in $\R^d$.\\
It is consistent with the previous assumptions to suppose that
\begin{itemize}
\item there  exists a map $\tilde H$ from $L^1 (\T^d)$ to  $\cC(\T^d \times \R^d)$
    such that  $H|_{\PP\cap L^1(\T^d)}= \tilde H|_{\PP\cap L^1(\T^d)}$ and that  for any $x\in \T^d$ and $ p\in \R^d$,
    $m\to \tilde H[m](x,p)$ is Fr{\'e}chet differentiable in $L^1 (\T^d)$ and $(x,p)\mapsto \frac {\partial \tilde H}{\partial m} [m](x,p)$
belongs to  $\cC^1(\T^d\times\R^d;  L^\infty(\T^d))$.
  We will not make the distinction between $H$ and $\tilde H$.
\item if $m\in \PP\cap L^1(\T^d)$ and   $a^*= \hbox{argmin}_a  f[m]( x,a)+ p\cdot g[m](x,a)$,  then
  \begin{displaymath}
    \frac {\partial H}{\partial m} (x,p) =  \frac {\partial f}{\partial m}(x,a^*)+p\cdot \frac {\partial g}{\partial m}(x,a^*).
  \end{displaymath}
\end{itemize}
As explained in \cite{MR3134900}, page 13,
 if the feedback function $v$ is smooth enough and if
 $m_0\in \PP\cap L^1(\T^d)$,  then
the probability distribution $m_{v}(t,\cdot)$ has a density with respect to the Lebesgue measure,  $m_{v}(t,\cdot)\in \PP\cap L^1(\T^d)$ for all $t$,
and  its density $m_v$ is solution of the Fokker-Planck equation
 \begin{equation}
\label{eq:4}
\frac{\partial m_v} {\partial t} (t,x)  - \nu \Delta m_v(t,x) +
\div\Big( m_v(t,\cdot) g[m_v (t,\cdot)] (\cdot, v(t,\cdot))  \Big) (x)=0,\;\; t\in (0,T], x\in \T^d,
 \end{equation}
with the initial condition
\begin{equation}
  \label{eq:5}
m_v(0,x)= m_0(x),\quad x\in \T^d.
\end{equation}
Therefore, the control problem consists of minimizing
\begin{displaymath}
  J(v,m_v)=\int_{\OT\times\T^d} f[m_v(t,\cdot)]( x, v(t,x)) m_v(t,x) dx dt + \int_{\T^d} h[m_v(T,\cdot)](x) m_v(T,x) dx,
\end{displaymath}
subject to (\ref{eq:4})-(\ref{eq:5}).
In \cite{MR3134900}, A. Bensoussan, J. Frehse and P. Yam have proved that a necessary condition
for the existence of a smooth feedback function $v^* $ achieving
$J( v^*, m_{v^*})= \min J(v, m_v)$ is that
\begin{displaymath}
  v^*(t,x)={\rm{argmin}}_{v}  \Bigl( f[ m(t,\cdot)](x,v)+ \nabla u(t,x)\cdot g[ m(t,\cdot)](x,v)\Bigr),
\end{displaymath}
where $(m,u)$ solve the
following system of partial differential equations
\begin{eqnarray}
  \label{eq:6}
0&=&
\begin{array}[t]{l}
\ds
\frac{\partial u} {\partial t} (t,x) + \nu \Delta u(t,x) + H[m(t,\cdot)]
( x, \grad u(t,x)) \\ \ds +
\int_{\T^d} \frac{\partial H} {\partial m}  [m(t,\cdot)]
(\xi, \grad u(t, \xi))(x) m(t,\xi) d\xi  ,
\end{array}
\\
\label{eq:7}
	0&=&\ds \frac{\partial m} {\partial t} (t,x)  - \nu \Delta  m(t,x) +
 \div\Bigl( m(t,\cdot) \frac{\partial H} {\partial p} [m(t,\cdot)](\cdot,\grad u(t,\cdot))\Bigr)(x) ,
\end{eqnarray}
with the initial and terminal conditions
\begin{equation}
  \label{eq:8}
  m(0,x)=m_0(x)\quad \hbox{and}\quad u(T,x) = h[m(T,\cdot)] (x) + \int_{\T^d}
\frac{\partial h} {\partial m} [m(T,\cdot)](\xi)(x) m(T,\xi) d\xi.
\end{equation}
It will be useful to write
\begin{equation}\label{eq:15}
  G[m,q](x) :=
 \int_{\T ^d} m(\xi) \frac {\partial}{\partial m} H [m]( \xi, q(\xi))(x)  d\xi
\end{equation}
for functions $m\in \PP\cap L^1(\T^d)$ and $q\in \cC (\T^d; \R^d) $, so that (\ref{eq:6}) can be written
\begin{displaymath}
  0=\frac{\partial u} {\partial t} (t,x) + \nu \Delta u(t,x) + H[m(t,\cdot)]
( x, \grad u(t,x))+  G[m(t,\cdot),\nabla u(t,\cdot)](x).
\end{displaymath}

\begin{remark}
  \label{sec:model-assumptions}
Note the difference with the system of partial differential equations arising in  mean field games, namely
\begin{eqnarray}
\label{eq:9}
0&=&
\ds
\frac{\partial u} {\partial t} (t,x) + \nu \Delta u(t,x) + H[m(t,\cdot)]
( x, \grad u(t,x)),
\\
\label{eq:10}
	0&=&\ds \frac{\partial m} {\partial t} (t,x)  - \nu \Delta  m(t,x) +
 \div\Bigl( m(t,\cdot) \frac{\partial H} {\partial p} [m(t,\cdot)](\cdot,\grad u(t,\cdot))\Bigr)(x) ,
\end{eqnarray}
with the initial and terminal conditions
\begin{equation}
\label{eq:11}
  m(0,x)=m_0(x)\quad \hbox{and}\quad u(T,x) = h[m(T,\cdot)] (x) .
\end{equation}

Both the HJB equation (\ref{eq:6}) and the terminal condition on $u$ in (\ref{eq:8}) involve additional nonlocal terms,
 which account for the variations of $m_v$ caused by variations of the common feedback $v$.
\end{remark}
\begin{remark}
  \label{sec:model-assumptions-1}
At least formally, it is possible to consider situations when  $H$ and $h$ depend
 locally on $m$, i.e. $H[m](x,p)= \tilde H(x, p, m(x))$ and $h[m](x)= \tilde h(x, m(x))$: in this case, (\ref{eq:6})-(\ref{eq:8})  become
\begin{eqnarray}
\label{eq:12}
0&=&
\begin{array}[t]{l}
\ds
\frac{\partial u} {\partial t} (t,x) + \nu \Delta u(t,x) + \tilde H
( x, \grad u(t,x),m(t,x) ) \\ \ds + m(t,x)
 \frac{\partial \tilde H} {\partial m}
(x, \grad u(t, x), m(t,x) ) ,
\end{array}
\\
\label{eq:13}
	0&=&\ds \frac{\partial m} {\partial t} (t,x)  - \nu \Delta  m(t,x) +
 \div\Bigl( m(t,\cdot) \frac{\partial \tilde H} {\partial p} (\cdot,\grad u(t,\cdot),m(t,\cdot))\Bigr)(x) ,
\end{eqnarray}
with the initial and terminal conditions
\begin{equation}
\label{eq:14}
  m(0,x)=m_0(x)\quad \hbox{and}\quad u(T,x) = \tilde h(x, m(T,x)) +
m(T,x) \frac{\partial \tilde h} {\partial m} (x,  m(T,x) ).
\end{equation}

\end{remark}
\section{Existence results}
\label{sec:existence-results}
We focus on the system (\ref{eq:6})-(\ref{eq:8}). We are going to state existence results in some typical situations.
\subsection{Notations}

Let $Q$ be the open set  $Q:=(0,T)\times \T^d$.  We shall need to use spaces of H{\"o}lder functions in $Q$:
For $\alpha\in (0,1)$, the space  of  H{\"o}lder functions $\cC^{ \alpha/2, \alpha} (\bar Q)$ is classically defined  by
\begin{displaymath}
  \cC^{\alpha/2, \alpha} (\bar Q):=\left \{
    \begin{array}[c]{l}
\ds w\in \cC(\bar Q) \;:\; \exists C>0 \hbox { s.t. } \forall (t_1,x_1), (t_2,x_2) \in \bar Q, \\
\ds |w(t_1,x_1)-w(t_2,x_2)| \le C\left( d(x_1,x_2)^2+ |t_1-t_2|   \right) ^{\alpha/ 2}
    \end{array}
\right \}
\end{displaymath}
and we define
\begin{displaymath}
  |w|_{\cC^{\alpha/2,\alpha } (\bar Q)}:=
\sup_{ (t_1,x_1)\not = (t_2,x_2) \in \bar Q}\frac {  |w(t_1,x_1)-w(t_2,x_2)|}{\left( d(x_1,x_2)^2+ |t_1-t_2|   \right) ^{\alpha/ 2} }
\end{displaymath}
and $\|w\|_{\cC^{\alpha/2,\alpha } (\bar Q)}:=  \|w\|_{\cC (\bar Q)}+  |w|_{\cC^{\alpha/2,\alpha }  (\bar Q)}$.
Then the space $\cC^{ (1+\alpha)/2, 1+\alpha} (\bar Q)$ is made of all the functions $w\in \cC(\bar Q)$
which have  partial derivatives $\frac{\partial w}{\partial x_i}\in \cC^{\alpha/2,\alpha } (\bar Q)$ for all $i=1,\dots, d$
and such that  for all $(t_1,x)\not = (t_2,x) \in \bar Q$,
    $ |w(t_1,x)-w(t_2,x)| \le C  |t_1-t_2|   ^{(1+\alpha)/ 2}$ for a positive constant $C$. The space
  $\cC^{ (1+\alpha)/2, 1+\alpha} (\bar Q)$, endowed with the semi-norm
  \begin{displaymath}
    |w|_{\cC^{ (1+\alpha)/2, 1+\alpha} (\bar Q)}:=
 \sum_{i=1}^{d} \|\frac{\partial w}{\partial x_i}\|_{\cC^{\alpha/2,\alpha } (\bar Q)} +
\sup_{ (t_1,x)\not = (t_2,x) \in \bar Q}\frac {  |w(t_1,x_1)-w(t_2,x_2)|}{  |t_1-t_2|   ^{(1+\alpha)/ 2} }
  \end{displaymath}
and norm $\|w\|_{\cC^{ (1+\alpha)/2, 1+\alpha} (\bar Q)}:=   \|w\|_{\cC (\bar Q)}+ |w|_{\cC^{ (1+\alpha)/2, 1+\alpha} (\bar Q)}$
is a Banach space.\\
Finally, the space $\cC^{ 1+\alpha/2, 2+\alpha}$ is made of all the functions $w \in \cC^1(\bar Q)$ which
are twice continuously differentiable w.r.t. $x$, with  partial derivatives
 $\frac{\partial w}{\partial x_i}\in \cC^{ (1+\alpha)/2,1+\alpha} (\bar Q)$ for all $i=1,\dots, d$, and
 $\frac{\partial w}{\partial t}\in \cC^{\alpha/2,\alpha } (\bar Q)$. It is a Banach space with the norm
 \begin{displaymath}
     \|w\|_{\cC^{1+\alpha/2, 2+\alpha} (\bar Q)}:=  \|w\|_{\cC (\bar Q)}+
 \sum_{i=1}^{d} \|\frac{\partial w}{\partial x_i}\|_{\cC^{ (1+\alpha)/2,1+\alpha} (\bar Q)}+
 \|\frac{\partial w}{\partial t}\|_{\cC^{\alpha/2,\alpha } (\bar Q)}
. \end{displaymath}
\subsection{The case when $\partial_p H$ is bounded}
\label{sec:case-when-partial_p}
We make the following assumptions on $h$, $m_0$, $H$ and $G$, in addition to the regularity assumptions on
 $H$ already made in \S~\ref{sec:model-assumptions}:
\begin{itemize}
\item[$(H_0)$] For simplicity only,  the map  $h$ is invariant w.r.t. $m$,
  i.e.  $h[m](x)=u_T(x)$, where $u_T$ is a smooth function defined on $\T^d$. Moreover,  $m_0$ is
  a smooth positive function.
\item[$(H_1)$] There exists a constant $\gamma_0>0$ such that
  \begin{equation*}
    |H[m](x,0)|
    \leq \gamma_0 \qquad  \forall  (m,x) \in  (\PP\cap L^1(\TTd)) \times \TTd
  \end{equation*}

\item[$(H_2)$]\label{condDpH} There exists a constant $\gamma_1>0$ such that
  \begin{equation*}
    \| \frac {\partial H}{\partial p}[m] \|_{{ \rm{Lip}} (\TTd\times  \RRd)  }
    \leq \gamma_1 \qquad  \forall  m \in  \PP\cap L^1(\TTd)
  \end{equation*}
\item[$(H_3)$] For all  $(m,x,p)\in (\PP\cap L^1(\TTd)) \times \TTd\times \R^d$,     $\frac {\partial H}{\partial m}[m](x,p)$ is a $\cC^1$ function on
$\TTd$ and there exists a constant $\gamma_2>0$ such that for all
  $(m,x,p)\in (\PP\cap L^1(\TTd)) \times \TTd\times \R^d$,
  \begin{equation*}	\| \frac {\partial H}{\partial m}[m](x, p)\|_{\cC^1(\T^d)}\le \gamma_2 (1+ |p|)
  \end{equation*}

\item[$(H_4)$]There exists a constant $\gamma_3>0$
 such that:
  \begin{equation*}
 % \| H[m_1](\cdot, 0)- H[m_2]( \cdot, 0)\|_{\cC^1(\TTd)} +
\| \frac {\partial H} {\partial p} [m_1](\cdot, 0)-  \frac {\partial H} {\partial p} H[m_2]( \cdot, 0)\|
_{\cC(\TTd)}
 \leq \gamma_3 \|m_1 - m_2\|_{L^1(\TTd)} \qquad \forall m_1,m_2 \in L^1(\TTd).
  \end{equation*}
\item[$(H_5)$] There exists $\gamma_4>0$ such that for $m_1,m_2\in \PP\cap L^1(\TTd)$, $p_1,p_2\in L^\infty(\TTd)$,
  \begin{equation*}
    \|G[m_1, p_1] - G[m_2,p_2]\|_{ L^\infty(\TTd)} \leq \gamma_4 \left( \|p_1 - p_2\|_{L^\infty(\TTd)} + \|m_1 - m_2\|_{L^1(\TTd)}\right).
  \end{equation*}
\end{itemize}

\paragraph{Example}  All the assumptions above are satisfied by the map $H$ :
\begin{displaymath}
 H[m] (x,p)= -\frac { \Phi(p)}{ (c+  (\rho_1* m) (x) )^\alpha} + F(x, (\rho_2*m) (x)),
\end{displaymath}
  where $\Phi$ is
a $\cC^2$ function from $\R^d$ to $\R_+$ such that $D^2 \Phi$ and $D \Phi$ are
bounded,   $\alpha$ and $c$ are positive numbers, $\rho_1$ and $\rho_2$ are  smoothing kernels in $\cC^\infty(\T^d)$,
$\rho_1$ is nonnegative, and $F$ is a $\cC^2$  function defined on $\T^d \times \R^d$. Here, $\rho*m(x) = \int_\TTd \rho(x-z)m(z)dz$.
It is easy to check that
\begin{equation*}
G[m, q](x) =
 \left(\alpha  \tilde \rho_1 * \left(m \frac {\Phi(q)} {(c+ \rho_1*m)^{\alpha+1}} \right)\right) (x)+
 \tilde \rho_2 * (m F'(\cdot, \rho_2*m ))(x)
\end{equation*}
where $\tilde \rho_1(x) = \rho_1(-x)$ and $\tilde \rho_2(x) = \rho_2(-x)$. \\
Such a Hamiltonian models situations in which there are congestion effects,
 i.e. the cost of displacement increases in the regions where the density is large.
The term $F(x, (\rho_2*m) (x))$  may model aversion to crowded regions. The prototypical situation is
$g[m](x,a)=a$ and $\Phi(q)= \min_{b\in \cK} (q\cdot b + \Phi^* (b))$, where $\cK$ is a compact subset of $\R^d$.
Setting $\Phi^*(b)=+\infty$ if $b\notin \cK$, $H$ corresponds to the cost
$f[m](x,a)=\frac 1 { (c+  (\rho_1* m) (x) )^\alpha} \Phi^* \left( a  (c+  (\rho_1* m) (x) )^\alpha \right) +  F(x, (\rho_2*m) (x))$.
 \subsubsection{A priori estimates}
\label{sec:priori-estimates} We first assume that (\ref{eq:6})-(\ref{eq:8}) has a sufficiently smooth solution
 and we look for a priori estimates.
\paragraph{ \bf Step 1: uniform bounds on $\|m\|_{L^p(0,T; W^{1,p}(\TTd))} + \|m\|_{\cC^{\alpha/2, \alpha}(\bar Q)}$, $p \in [1, \infty),  \alpha \in [0,1)$ }

 First, standard arguments yield that  $m(t,\cdot)\in \PP$ for all $t\in [0,T]$.\\
From Assumption $(H_2)$, the function
$b: (t,x)\mapsto \partial_p H[m(t,\cdot)]( x,\grad u(t,x))$ is such that
 $\|b\|_{L^\infty(Q)}\le \gamma_1$. The Cauchy problem satisfied by $m$ can be written
 \begin{equation}
   \label{eq:18}
   \begin{split}
    \ds \frac{\partial m} {\partial t} (t,x)  - \nu \Delta  m(t,x) +\div  (b(t,\cdot) m(t,\cdot)) (x)=0 ,\\
    m(0,x)=m_0(x),
   \end{split}
 \end{equation}
and from the classical theory on weak solutions to parabolic equations, see e.g. Theorem 6.1 in \cite{MR1465184}, there exists a constant $C_0$ depending only on $\|m_0\|_{L^2(\T^d)}$ such that
\[\|m\|_{ L^2\big(0,T; H^1(\TTd)\big)} +
\|m\|_{\cC\big(\OT; L^2(\TTd)\big)} \le C_0.\]
 Moreover, since the operator in (\ref{eq:18}) is in divergence form,
we have  maximum estimates on $m$, see   Corollary 9.10 in \cite{MR1465184}:
there exists a constant $C_1$ depending only on $\|m_0\|_\infty$ and $\gamma_1$ such that
\begin{equation}\label{eq:19}
	m(t,x)\leq C_1 \qquad \forall (t,x)\in\OT\times\T^d.
\end{equation}
Therefore, the Fokker-Planck equation in (\ref{eq:18}) can be rewritten
\begin{equation}
  \label{eq:20}
      \ds \frac{\partial m} {\partial t} (t,x)  - \nu \Delta  m(t,x) +\div  (B(t,\cdot) ) (x)=0 ,
\end{equation}
where $\| B\| _{\infty} \le   \gamma_1 C_1$.
 From from standard results on the heat equation, see \cite{MR0241822},  this implies that for all $p\in [1,\infty)$ there exists a constant $C_2(p)$ which
 depends on $\|m_0\|_{\infty}$ and $\gamma_1$,
 such that
 \begin{equation}\label{eq:21}
 	\|m\|_{L^p\big(0,T; W^{1,p}(\TTd) \big)}
  + \|\frac {\partial m}{\partial t}\|_{L^p\big(0,T; W^{-1,p}(\TTd) \big)}
 \leq C_2(p) .
 \end{equation}
Finally,
 H{\"o}lder estimates for the heat equation with a right hand side
in divergence form, see for example Theorem 6.29  in \cite{MR1465184}, yield that for any $\alpha\in (0,1)$, there exists a positive constant
$C_3(\alpha)\ge C_1$ which only depends  on $\gamma_1$ and on $\|m_0\|_{\cC^\alpha (\T^d)}$  such that
\begin{equation}\label{eq:22}
	\|m\|_{\cC^{\alpha/2,\alpha}(\bar Q)} \leq C_3(\alpha).
\end{equation}
\paragraph{\bf Step 2: uniform bounds on $\|u\|_{\cC^{(1+\theta)/2, 1+\theta}(\bar Q)}$, $ \theta \in (0,1)$}
Defining
\begin{displaymath}
  a(t,x):=  -H[m(t,\cdot)](x,0)\quad \hbox{and }\quad
 A (t,x):=\int_0^1 \frac {\partial H}{\partial p} [m(t,\cdot)](x, \zeta \nabla u (t,x)) d\zeta,
\end{displaymath}
 the HJB equation (\ref{eq:6}) can be rewritten
\begin{equation}\label{eq:23}
	\frac{\partial u}{\partial t}(t,x) + \nu \Delta u(t,x)
+  A(t,x)  \cdot \grad u(t,x) =  a(t,x) - G[m(t,\cdot), \grad u (t,\cdot)] (x).
\end{equation}
For some smooth function $\hat u$, let us consider
\begin{equation}\label{eq:24}
	\frac{\partial u}{\partial t}(t,x) + \nu \Delta u(t,x)
+  A(t,x)  \cdot \grad u(t,x) =  a(t,x) - G[m(t,\cdot), \grad \hat u (t,\cdot)] (x)
\end{equation}
instead of (\ref{eq:23}), with the same terminal condition as in (\ref{eq:8}).
From Assumption  $(H_1)$ and  $(H_2)$,  $\|a\|_{\infty}\le \gamma_0 $ and $\|A\|_{\infty}\le \gamma_1$.
From Assumption   $(H_3)$,
\begin{equation}
  \label{eq:25}
\|G[m, \grad \hat u ]\|_{L^2(\T^d)}\le c ( 1+ \|\nabla \hat u\|_{L^2(\T^d)}   ),
\end{equation}
 where $c>0$ depends on  $C_1$ in (\ref{eq:19}) and $\gamma_2$.
Multiplying (\ref{eq:24}) by $u(t,x)e^{-2\Lambda t}$ and integrating on $\T^d$,
then using the bounds on  $\|a\|_{\infty}$, $\|A\|_{\infty}\le \gamma_1$ and
  (\ref{eq:25}), a standard argument yields that there exist constants $\Lambda$  and $\tilde C_4$  which depend only on $\gamma_0$, $\gamma_1$, $\gamma_2$,  $\|m_0\|_{\infty}$ such that
  \begin{equation}
    \label{eq:26}
    \begin{split}
&     -\frac {d}{dt}  \left( \|u(T-t,\cdot)\|^2_{L^2(\T^d)} e^{-2\Lambda (T-t)}\right) +  \nu  \|\nabla u(T-t,\cdot)\|^2_{L^2(\T^d)} e^{-2\Lambda (T-t)}
 \\ \le & \tilde C_4 + \frac \nu  2  \|\nabla \hat u(T-t,\cdot)\|^2_{L^2(\T^d)} e^{-2\Lambda (T-t)} .
    \end{split}
  \end{equation}
Hence, if
\begin{equation}\label{eq:27}
   \nu   \int_{t=0}^T \|\nabla \hat u(T-t,\cdot)\|^2_{L^2(\T^d)} e^{-2\Lambda (T-t)} dt \le
C_4,
\end{equation}
with
\begin{equation}\label{eq:28}
C_4= 2 \tilde C_4 T + 2 \int_{\T^d}  u_T^2 (x) dx,
\end{equation}
then
\begin{equation}\label{eq:29}
   \sup_{t}  e^{-2\Lambda (T-t)} \int_{\T^d}  u^2 (T-t, x) dx +
 \nu   \int_{t=0}^T \|\nabla  u (T-t,\cdot)\|^2_{L^2(\T^d)} e^{-2\Lambda (T-t)} dt \le C_4.
\end{equation}
 Similarly,  a solution  of (\ref{eq:6})-(\ref{eq:8}) satisfies (\ref{eq:29}) with the same constants $\Lambda$ and $C_4$.
Note that $\Lambda$ can be chosen large enough such that the function $(t,x)\mapsto u_T(x)$ satisfies (\ref{eq:29}).

For a solution  of (\ref{eq:6})-(\ref{eq:8}), this implies that    $\partial_t u + \nu \Delta u$
is bounded in $L^2(Q)$,
hence that $u$ is bounded in $\cC^0(0,T; H^1(\TTd))$
by a constant  $\bar C_4> \|u_T\|_{H^1(\T^d)}$ which depends  on $ \Lambda$, $C_4$, $\gamma_1$ and
 $\|u_T\|_{H^1(\T^d)}$, i.e.
\begin{equation}
  \label{eq:31}
\| \grad u \|_{L^\infty (0,T; H^1(\T^d))} \le \bar C_4.
\end{equation}
\\
As a consequence, the left-hand side of (\ref{eq:23}) is bounded in $L^\infty(Q)$,
and this yields   H{\"o}lder estimates on $u$: by using Theorem 6.48 in \cite{MR1465184},
we see that for all $\theta\in (0,1)$, there exists a constant $C_5(\theta)$ which depends on
$\theta$, $\|m_0\|_{\infty}$, $\|u_T\|_{\cC^{1+\theta}(\T^d)}$,  $\gamma_0$, $\gamma_1$,  $\gamma_2$ such that
 \begin{equation} \label{eq:32}
	\|u\|_{\cC^{(1+\theta)/2, 1+\theta } (\bar Q) } \leq C_5(\theta),
\end{equation}
which holds  for a solution of (\ref{eq:24})  with the terminal condition (\ref{eq:8}), as soon as
$\hat u$ satisfies (\ref{eq:29}) and (\ref{eq:31}).

\paragraph{\bf Step 3: uniform bound on $ \|m\|_{\cC^{(1+\theta)/2,1+\theta}  (\bar Q) }$,  $ \theta \in (0,1) $} Let us go back to (\ref{eq:7}).
From  Assumptions $(H_1)-(H_4)$, and from the previous two steps, we see that for any $ \theta \in (0,1) $,
$m$ and $\frac{\partial H} {\partial p} [m](\cdot,\grad u)$ are both  bounded in $\cC^{\theta/2, \theta } (\bar Q) $
 by constants which depend on $m_0$ and $u_T$,  and  $\gamma_0,\dots,\gamma_3$.
Thus, the function $B$ in  (\ref{eq:20}) is bounded in $\cC^{\theta/2, \theta } (\bar Q) $. Using  Theorem 6.48 in \cite{MR1465184} for the heat equation with a data in divergence form,
we see that for all $\theta\in (0,1)$, there exists a constant $C_6(\theta)$ which depends on
$\theta$, $\|m_0\|_{\cC^{1+\theta}(\T^d)}$, $\|u_T\|_{\cC^{1+\theta}(\T^d)}$,  $\gamma_0,\dots,\gamma_3$ such that
 \begin{equation*}
	\|m\|_{\cC^{(1+\theta)/2, 1+\theta } (\bar Q) } \leq C_6(\theta).
\end{equation*}
\paragraph{\bf Step 4: uniform bounds on $\|u\|_{\cC^{1+\theta/2, 2+\theta}(\bar Q)}$, $ \theta \in (0,1)$}
From the previous steps and Assumptions $(H_1)-(H_4)$, we see that there exists a constant $c$ such that the functions in (\ref{eq:23}) satisfy
  $\|a\|_{\cC^{\theta/2, \theta } (\bar Q) }  \le c $ and $\|A\|_{\cC^{\theta/2, \theta } (\bar Q) } \le c$.
Similarly, from Assumptions $(H_3)$ and $(H_5)$, $\|G[m, \grad u ]\|_{\cC^{\theta/2, \theta } (\bar Q) } \le c$.
Standard regularity results on parabolic equations, for instance Theorem 4.9 in \cite{MR1465184} lead to the existence of  $C_7(\theta)$
such that
\begin{equation*}
	\|u\|_{\cC^{1+\theta/2, 2+\theta}(\bar Q)} \leq C_7(\theta).
\end{equation*}

\subsubsection{The existence theorem}
\label{sec:existence-theorem}
\begin{theorem}\label{sec:existence-theorem-2}
  Under the Assumptions $(H_0)$-$(H_5)$, for $\alpha\in (0,1)$ there exist functions $u\in \cC^{1+\alpha/2,2+\alpha}(\bar Q)$ and
$m\in \cC^{(1+\alpha)/2,1+\alpha}(\bar Q)$
% \cap L^p(0,T; W^{1,p}(\T^d))$ for all $p\in [1,+\infty)$
 which satisfy  (\ref{eq:6})-(\ref{eq:8}), ( note that (\ref{eq:7}) is satisfied in a weak sense).
\end{theorem}

\begin{proof}
  The argument is reminiscent of that used  by J-M. Lasry and P-L. Lions for mean field games: it is done in two steps

  \paragraph{Step A}
For $R>0$, let $\eta_R: \R\to \R$ be a smooth, nondecreasing  and odd function such that
\begin{enumerate}
\item $\eta_R(y)=y$ if $|y|\le R$, $\eta_R(y)= 2R$ if  $y\ge 3R$
\item $\|\eta_R'\|_{\infty}\le 1$
\end{enumerate}
We consider the modified set of equations
\begin{eqnarray}
  \label{eq:79}
0=
\ds
\frac{\partial u} {\partial t} (t,x) + \nu \Delta u(t,x) + H[m(t,\cdot)]
( x, \grad u(t,x))+  \eta_R (G[m(t,\cdot),\nabla u(t,\cdot)](x)),
\\
\label{eq:80}
	0=\ds \frac{\partial m} {\partial t} (t,x)  - \nu \Delta  m(t,x) +
 \div\Bigl( m(t,\cdot) \frac{\partial H} {\partial p} [m(t,\cdot)](\cdot,\grad u(t,\cdot))\Bigr)(x) .
\end{eqnarray}

We are going to apply Leray-Shauder fixed point theorem to a map $\chi$ defined
for example in $ X=\left \{  m\in \cC^0 ([0,T]; L^2 (\T^d) \cap \PP)\right \}  $:
  consider first the map $\psi:  X \to X \times L^2(0,T; H^1(\T^d))$, $m\mapsto (m,u)$
where $u$ is a weak solution of  (\ref{eq:79}) and $u|_{t=T}=u_T$.  Existence and uniqueness for this problem are well known.
Moreover, from the estimates above,  for every $0<\alpha<1$, $\|u\|_{\cC^{1/2+\alpha/2,1+\alpha }(\bar Q)}$
 is  bounded by a constant independent of $m$ and $m\mapsto u$ is continuous from $  X$ to
$\cC^{1/2+\alpha/2,1+\alpha }(\bar Q) $.
\\
Fix $\theta\in (0,1)$, and consider the map  $\zeta :  X\times \cC^{1/2+\theta/2,1+\theta}(\bar Q)  \to  L^2(0,T; H^1(\T^d))$,
$(m,u)\mapsto \tilde m$ where $\tilde m$ is a weak solution of  the Fokker-Planck equation
\begin{displaymath}
  	0=\ds \frac{\partial \tilde m} {\partial t} (t,x)  - \nu \Delta  \tilde m(t,x) +
 \div\Bigl( \tilde m(t,\cdot) \frac{\partial H} {\partial p} [m(t,\cdot)](\cdot,\grad u(t,\cdot))\Bigr)(x) .
\end{displaymath}
 and $\tilde m|_{t=0}=m_0$. Existence and uniqueness are well known, and moreover, the estimates above tell us that for all $0<\alpha<1$,
there exists $ R_\alpha>0$ such that
$\|\tilde m \|_{\cC^{\alpha/2,\alpha }(\bar Q)}  \le R_\alpha  $ uniformly with respect to $m$ and $u$.
Moreover from the assumptions, it can be seen that
$\zeta$ maps continuously  $X\times \cC^{1/2+\theta/2,1+\theta}(\bar Q) $ to $X$.
\\
Let $K$ be the set $\{\|m \|_{\cC^{\alpha/2,\alpha }(\bar Q)}  \le R_\alpha;  m|_{t=T}=m_T \}\cap X$:
this set is a compact and convex subset of $X$
 and the map $\chi=\zeta\circ \psi$: $m\mapsto \tilde m$ is continuous in $X$ and leaves $K$ invariant.
We can apply Leray-Shauder fixed point theorem  the map $\chi$,
 which yields the existence of a solution $(u_R,m_R)$ to (\ref{eq:79})-(\ref{eq:80}). Moreover the a priori estimates above
 tell us that  $u_R\in \cC^{1+\alpha/2,2+\alpha}(\bar Q)$ and
$m_R\in \cC^{(1+\alpha)/2,1+\alpha}(\bar Q)$.% \cap L^p(0,T; W^{1,p}(\T^d))$ for all $p\in [1,+\infty)$.

\paragraph{Step B}
Looking at all the a priori estimates above, it can be seen that $m_R$, (resp $u_R$)
 belongs to a bounded subset of $\cC^{\alpha/2,\alpha }(\bar Q)$ (resp.  $\cC^{1/2+\alpha/2,1+\alpha }(\bar Q)$)
independent of $R$. Hence, for $R$ large enough, $ \eta_R (G[m_R,\nabla u_R])= G[m_R,\nabla u_R]$, and $(m_R,u_R)$ is a weak solution of   (\ref{eq:6})-(\ref{eq:8}), with    $u_R\in \cC^{1+\alpha/2,2+\alpha}(\bar Q)$ and
$m_R\in \cC^{(1+\alpha)/2,1+\alpha}(\bar Q) $.
\end{proof}
\begin{remark}
  \label{sec:existence-theorem-1}
It is possible to weaken some of the assumptions in Theorem  \ref{sec:existence-theorem-2}: for example, we can assume the following weaker  version of $(H_2)$, namely:\\
$(H_2')$ There exists a constant $\gamma_1>0$ and $ \eta \in (0,1)$ such that
\begin{itemize}
\item   $\forall  m \in  \PP\cap L^1(\TTd)$,  $\| \frac {\partial H}{\partial p}[m] \|_{ \cC (\TTd\times  \RRd)  }   \leq \gamma_1$
\item  $\forall  m \in  \PP\cap L^1(\TTd)$,  $x,y\in \T^d$, $p,q\in \R^d$,
  \begin{displaymath}
|    \frac {\partial H}{\partial p}[m](x,p)-\frac {\partial H}{\partial p}[m](y,q)| \le \gamma_1 ( d(x,y) + | p-q|^\eta)
  \end{displaymath}
\end{itemize}
Indeed, the regularity of $ \frac {\partial H}{\partial p}$ with respect to $p$ is only used in Steps 3 and 4 above: with this weaker assumptions, the conclusions of steps 3 and 4 hold with $0<\theta<\eta$, and this is enough for proving the existence of
$u\in \cC^{1+\alpha/2,2+\alpha}(\bar Q)$ and
$m\in \cC^{(1+\alpha)/2,1+\alpha}(\bar Q) $
 for some $\alpha$, $0<\alpha<\eta $
 which satisfy  (\ref{eq:6})-(\ref{eq:8}).
\end{remark}

\subsection{ Hamiltonian with a subquadratic growth in $p$: a specific case}
\label{sec:hamilt-with-subq}
For a smooth nonnegative periodic function $\rho$, two  constants $\alpha>0$ and $\beta$, $1<\beta\le 2$,
 let us focus on  the following Hamiltonian:
\begin{equation}\label{eq:37}
	H[m](x, p) :=- \frac{|p|^\beta}{ (1+(\rho*m)(x))^\alpha}.
\end{equation}
 The map $G$ defined in (\ref{eq:15}) is
\begin{equation*}
	G[m, q](x) = \alpha \left(\tilde \rho * \left(m \frac {|q|^\beta} {\left(1+(\rho*m)\right)^{\alpha+1}} \right)\right)(x),
\end{equation*}
where $\tilde \rho(x) := \rho(-x)$.
\\
Assuming that $m_0$ is smooth, let us call $\bar m_0=\|m_0\|_{\infty}$:  for all $x\in \T^d$, $0< m_0(x)\le \bar m_0$.
We assume that
\begin{equation}
  \label{eq:38}
\|\rho\|_{L^1(\T^d)} < \frac{\beta-1}{ \alpha \bar m_0}.
\end{equation}

\begin{remark}
 \label{sec:hamilt-with-subq-1}
It would be interesting to make further investigations to see if
the assumption on the regularizing kernel  $\rho$ in (\ref{eq:38}) is really necessary, since it is not necessary in the context of mean field games with congestion.
Yet, in the a priori estimates proposed below, (\ref{eq:38}) is useful for getting a bound on $\|m H[m](\cdot,\nabla u)\|_{L^1 (Q)}$, see (\ref{eq:40}).
\end{remark}

\subsubsection{A priori estimates}
\label{sec:priori-estimates-1}
We first assume that (\ref{eq:6})-(\ref{eq:8}) has a sufficiently smooth weak solution
 and we look for a priori estimates.
\paragraph{\bf Step 1: a lower bound on $u$}
Since $G$ is non negative,  by comparison, we see that
\begin{equation*}
	u(t,x) \ge \underline u_T:= \min_{\xi \in \TTd} u(T,\xi) \qquad \forall (t,x) \in \OT \times \TTd.
\end{equation*}
\paragraph{\bf Step 2: an energy estimate and its consequences}
Let us multiply (\ref{eq:6}) by $m-\bar m_0$ and (\ref{eq:7})
 by $u$ and integrate the two resulting equations on $\T^d$.
Summing the resulting identities, we obtain:
 \begin{align*}
   &\int_Q \frac{ \partial}{\partial t} (u(t,x)(m(t,x)-\bar m_0))dxdt +
   \int_Q H[m(t, \cdot)]( x, \grad u (t,x)) (m(t,x)-\bar m_0)dxdt\\
   & + \int_Q G[m(t,\cdot),\grad u(t,\cdot)]( x) (m(t,x)-\bar m_0)dxdt \\
&+
{\int_Q \div\left( m(t,x) \frac{ \partial}{\partial p} H[m(t,\cdot)]( x, \grad u(t,x))
\right) u(t,x)dxdt}
= 0
\end{align*}
Hence
\begin{equation}
  \label{eq:39}
  \begin{split}
&\int_{\T^d} u(T,x)(m(T,x)-\bar m_0) dx + \int_{\T^d} u(0,x)(\bar m_0-m(0,x)) dx\\
	= &\int_Q H[m(t, \cdot)]( x, \grad u (t,x)) (\bar m_0-m(t,x))dxdt\\
 &+ \int_Q G[m(t,\cdot),\grad u(t,\cdot)]( x) (\bar m_0-m(t,x))dxdt\\
&+\int_Q m(t,x) \frac{ \partial}{\partial p} H[m(t,\cdot)]( x, \grad u(t,x))\cdot \grad u(t,x) dxdt
  \end{split}
\end{equation}
In  (\ref{eq:39}), the first term in the left hand side is bounded from below by $ -\|u_T\|_{\infty}( 1+ \bar m_0)$,
The second term is larger than $   \underline u_T\int_{\T^d} (\bar m_0-m(0,x)) dx= (\bar m_0 -1) \underline u_T$. Therefore, the
left hand side of (\ref{eq:39}), is bounded from below by a constant $c$ which only depends on $ \bar m_0$ and $ \|u_T\|_{\infty}$; we obtain that
\begin{displaymath}
    \begin{split}
c\le  & (\beta -1) \int_Q     m(t,x) H[m(t, \cdot)]( x, \grad u (t,x)) dxdt + \int_Q \bar m_0 H[m(t, \cdot)]( x, \grad u (t,x)) dxdt
\\
& + \alpha\int_Q(\bar m_0-m(t,x)) \tilde\rho * \left(m(t,\cdot) \frac {|\grad u(t,\cdot)|^\beta} {(1+\rho*m(t,\cdot))^{\alpha+1}}\right)(x) dxdt.
  \end{split}
\end{displaymath}
We see that  last term can be bounded as follows:
\begin{displaymath}
  \begin{split}
       & \int_Q(\bar m_0-m(t,x)) \tilde\rho * \left(m(t,\cdot) \frac {|\grad u(t,\cdot)|^\beta} {(1+\rho*m(t,\cdot))^{\alpha+1}}\right)(x) dxdt
\\ \le &
\bar m_0\int_Q \tilde\rho * \left(m(t,\cdot) \frac {|\grad u(t,\cdot)|^\beta} {(1+\rho*m(t,\cdot))^{\alpha}}\right)(x) dxdt
\\
\le & \bar m_0 \|\rho\|_{L^1(\T^d)} \int_Q m(t,x) \frac {|\grad u(t,x)|^\beta} {(1+\rho*m(t,x))^{\alpha}} dxdt
\\ = &  - \bar m_0 \|\rho\|_{L^1(\T^d)} \int_Q  m(t,x) H[m(t, \cdot)]( x, \grad u (t,x)) dxdt .
   \end{split}
\end{displaymath}
Therefore,
\begin{displaymath}
    \begin{split}
 c\le & \ds  \left (\beta -1  - \alpha \bar m_0 \|\rho\|_{L^1(\T^d)}\right ) \int_Q     m(t,x) H[m(t, \cdot)]( x, \grad u (t,x)) dxdt \\ & \ds + \int_Q \bar m_0 H[m(t, \cdot)]( x, \grad u (t,x)) dxdt.
\end{split}
\end{displaymath}
From (\ref{eq:37}) and (\ref{eq:38}), we see that there exists a constant $C_1$ which depends on $ \bar m_0$ and $ \|u_T\|_{\infty}$
such that
\begin{equation}
  \label{eq:40}
\| m H[m](\cdot, \nabla u )\|_{L^1(Q)} + \|  H[m](\cdot, \nabla u )\|_{L^1(Q)} \le C_1.
\end{equation}
Using (\ref{eq:40}), we deduce from a comparison argument applied to the HJB equation
that there exists a constant $C_2$  which depends on $ \bar m_0$ and $ \|u_T\|_{\infty}$
such that
\begin{equation}
  \label{eq:41}
\|u\|_{L^\infty(Q)}\le C_2.
\end{equation}
Since $1< \beta \le 2$, there exists a constant $c$ such that
$|\frac{\partial H[m]}{\partial p} (x, p)|^2\le c(1 - H[m](x,p))$. We deduce from (\ref{eq:40}) and the latter observation that there exists a constant $C_3>0$ such that
\begin{equation}
\label{eq:42}
\int_Q  (m(t,x)+ 1) \left |\frac{\partial H[m(t,\cdot)]}{\partial p}(x, \nabla u(t,x) )\right|^2dxdt
 \le C_3.
\end{equation}
\paragraph{\bf Step 3: uniform estimates from the Fokker-Planck equation}
The following estimates can be proved exactly  as in \cite{MR3195848},  Lemma 2.3 and Corollary 2.4,  (see also \cite{MR2928380}, Lemma 2.5 and Corollary 2):
\begin{lemma} \label{sec:bf-step-3} For  $\gamma=\frac {d+2}d$ if $d>2$ and all $\gamma<2$ if $d=2$, there exists an  constant $c>0$, (independent from $m_0$ and $u_T$)
 such that
  \begin{equation}\label{eq:43}
    \begin{split}
 & \sup_{t\in [0,T]} \|m(t,\cdot )\log( m(t,\cdot ))\|_{L^1(\T ^d)} +
 \| \sqrt m \|_{L^2(0,T; H^1(\T^d))}^2+ \|m\|_{L^\gamma(Q)}^\gamma \\
\le &c\left( \int_Q  m(t,x)\left |\frac{\partial H[m(t,\cdot)]}{\partial p}(x, \nabla u(t,x) ) \right|^2 dxdt
+ \int_{\T^d} m_0(x)\log(m_0(x))dx \right)  .
    \end{split}
  \end{equation}
\end{lemma}
\begin{corollary}
For $q= \frac {d+2}{d+1}$ if $d>2$ and $q<4/3$ if $d=2$,
 there exists a constant $c>0$ such that
 \begin{equation}\label{eq:44}
   \begin{split}
  &   \|\nabla m\|_{L^q (Q)}^q+ \| \frac {\partial m}{\partial t}\|^q_{L^q (0,T; W^{-1,q} (\T^d))} \\
\le & c\left( \int_Q  m(t,x)\left |\frac{\partial H[m(t,\cdot)]}{\partial p}(x, \nabla u(t,x) ) \right|^2 dxdt  + \int_{\T^d} m_0(x)\log(m_0(x))dx \right) .
   \end{split}
 \end{equation}
\end{corollary}
From (\ref{eq:44}) and (\ref{eq:42}), we have a  uniform bound on $\| \frac {\partial m}{\partial t}\|^q_{L^q (0,T; W^{-1,q} (\T^d))}$
by a constant
depending only on $u_T$ and $m_0$.
We infer  that (\ref{eq:6}) can be written
\begin{equation}
  \label{eq:45}
\frac{\partial u} {\partial t} (t,x) + \nu \Delta u(t,x) + a(t,x )|\nabla u|^\beta(t,x)= b(t,x),
\end{equation}
where $a$ is a function which belongs to $\cC([0,T]; \cC^p(\T^d))$ for all $p\in \N$, with corresponding norms bounded by constants
depending only on $u_T$ and $m_0$. From (\ref{eq:40}), we deduce that for all $p\in \N$,
 $\|b\|_{L^1(0,T; W^{p,\infty} (\T^d))}$ is bounded by a constant
depending only on $u_T$ and $ m_0$, because
\begin{displaymath}
  \begin{split}
  \|b\|_{L^1(0,T; W^{p,\infty} (\T^d) )}&= \|G[m,\nabla u] \|_{L^1(0,T; W^{p,\infty} (\T^d))}\\
  &\le
  c\left \| \frac {m |\nabla u|^\beta}{\left(1+(\rho*m)\right)^{\alpha+1}}\right \|_{L^1(Q)}    \\
 &\le c \| m H[m](\cdot, \nabla u )\|_{L^1(Q)} .
  \end{split}
\end{displaymath}

\paragraph{\bf Step 4: uniform estimates on $|\nabla u|$}
Since $a \in \cC([0,T]; \cC^p(\T^d))$ and\\  $b\in L^1(0,T; W^{p,\infty} (\T^d))$, we can apply  Bernstein method to~(\ref{eq:45}) and
 estimate $|\nabla u|$.
By a slight modification of the proof  of Theorem 11.1 in  \cite{MR1465184},
(the only difference is that in  \cite{MR1465184}, $b$ is supposed to belong to $L^\infty(0,T; W^{p,\infty} (\T^d))$, but it can be checked that
this assumption can be weakened), we prove that there exists a constant $C_4$ which depends on
 $u_T$ and $ m_0$ such that
 \begin{equation}
   \|\nabla u \|_{L^\infty (Q)}\le C_4.
 \end{equation}
The proof adapted from \cite{MR1465184} is rather long, so we do not reproduce it here.

\paragraph{\bf Step 5: stronger a priori estimates}
Since $|\nabla u|$ is bounded, we can recover all the a priori estimates in \S~\ref{sec:priori-estimates},
except that the estimates in Step 3 and 4 of \S~\ref{sec:priori-estimates} only hold with $0<\theta< \beta-1$, in view of Remark \ref{sec:existence-theorem-1}.
We obtain that for all $\gamma\in (0,1)$, there exist two constants $C_5(\gamma)$ and $C_6(\gamma)$ such that
 $  \| m\| _{\cC^{\gamma/2, \gamma} (\bar Q)}\le C_5(\gamma)$ and $\|u\|_{\cC^{(1+\gamma)/2, 1+\gamma} (\bar Q)}\le C_5(\gamma)  $,
and that for all $\theta\in (0,\beta-1)$, there exist two constants $C_7(\theta)$ and $C_8(\theta)$ such that
 $\|m\|_{\cC^{(1+\theta)/2, 1+\theta} (\bar Q)}\le C_7 (\theta)  $ and $  \| u\| _{\cC^{1+ \theta/2,2+ \theta} (\bar Q)}\le C_8(\theta)$.

\subsubsection{The existence theorem}
\label{sec:existence-theorem-3}
\begin{theorem}\label{sec:existence-theorem-4}
We assume $(H_0)$ and  (\ref{eq:38}).
 For $\gamma$, $0<\gamma < \beta-1$,  there exists a function $u\in \cC^{1+\gamma/2,2+\gamma}(\bar Q)$ and
$m\in \cC^{(1+\gamma)/2,1+\gamma}(\bar Q) $
 which satisfy  (\ref{eq:6})-(\ref{eq:8}) with  $H$ given by (\ref{eq:37}).
\end{theorem}
\begin{proof}
We start by suitably truncating the Hamiltonian $H$ and the map $G$: for $R>1$, define
\begin{equation}
  \label{eq:72}
H_R[m](x,p)=\left\{
  \begin{array}[c]{ll}
\ds - \frac{|p|^\beta}{ (1+(\rho*m)(x))^\alpha} \quad    &\hbox{if } |p|<R,
\\
 \ds - \frac{  \beta R^{\beta -1} |p|  + (1-\beta) R^\beta}{ (1+(\rho*m)(x))^\alpha} \quad    &\hbox{if } |p|\ge R,
  \end{array} \right.
\end{equation}
and
\begin{equation}\label{eq:73}
  G_R[m, q](x) = \alpha \left(\tilde \rho * \left(m \frac {\min (|q|^\beta,R^\beta)} {\left(1+(\rho*m)\right)^{\alpha+1}} \right)\right)(x).
\end{equation}
Note that
\begin{equation}\label{eq:78}
-H_R[m](x,p)+ \frac{\partial}{\partial p} H_R[m](x,p)\cdot p= \left\{
  \begin{array}[c]{ll}
\ds    -(\beta -1) \frac{|p|^\beta}{ (1+(\rho*m)(x))^\alpha}\quad  &\hbox{if } |p|\le R,\\
   \ds  -(\beta -1) \frac{R^\beta}{ (1+(\rho*m)(x))^\alpha}\quad  &\hbox{if } |p|\ge R.
  \end{array}
\right.
\end{equation}
Thanks to Remark~\ref{sec:existence-theorem-1}, we can use a slightly modified version of
 Theorem \ref{sec:existence-theorem-2}: for some  $\gamma$, $0<\gamma < \beta-1$,
there exists a solution $(u_R,m_R)$ of
\begin{eqnarray*}
\label{eq:74}
0=
\ds
\frac{\partial u_R} {\partial t} (t,x) + \nu \Delta u_R(t,x) + H_R[m_R(t,\cdot)]
( x, \grad u_R(t,x))+  G_R[m_R(t,\cdot),\nabla u_R(t,\cdot)](x),
\\
\label{eq:75}
	0=\ds \frac{\partial m_R} {\partial t} (t,x)  - \nu \Delta  m_R(t,x) +
 \div\Bigl( m_R(t,\cdot) \frac{\partial H_R} {\partial p} [m_R(t,\cdot)](\cdot,\grad u_R(t,\cdot)\Bigr)(x) ,
\end{eqnarray*}
with the initial and terminal conditions (\ref{eq:8}), such that   $u_R\in \cC^{1+\gamma/2,2+\gamma}(\bar Q)$ and
$m_R\in \cC^{(1+\gamma)/2,1+\gamma}(\bar Q) $. %\cap L^p(0,T; W^{1,p}(\T^d))$ for all $p\in [1,+\infty)$
\\
Then it is possible to carry out the same program as in Step 1  and 2 in \S~\ref{sec:priori-estimates-1}:
 using  (\ref{eq:72})-(\ref{eq:78}), we  obtain
that there exists a constant $c$ independent of $R$ such that
\begin{displaymath}
    \begin{split}
 c\le & \ds  -\left (\beta -1  - \alpha \bar m_0 \|\rho\|_{L^1(\T^d)}\right ) \int_Q
   m_R(t,x) \frac  {|\nabla u_R(t,x)|^\beta}{(1+\rho*m_R(t,x))^{\alpha}}
   1_{\{|\nabla u_R(t,x)| <R\}} dxdt \\
& \ds - \left (\beta -1  - \alpha \bar m_0 \|\rho\|_{L^1(\T^d)}\right ) \int_Q
m_R(t,x) \frac  {R^\beta}{(1+\rho*m_R(t,x))^{\alpha}}    1_{\{|\nabla u_R(t,x)| \ge R\}} dxdt  \\
 & \ds + \int_Q \bar m_0 H_R[m_R(t, \cdot)]( x, \grad u_R (t,x)) dxdt,
\end{split}
\end{displaymath}
and this implies
 the counterpart of (\ref{eq:40}):  there exists a constant $C$ independent of $R$ such that
\begin{equation}
\label{eq:76}
\left \|  (1+m_R)    \frac {\min (|\nabla u_R|^\beta, R^\beta)}{ (1+(\rho*m_R))^\alpha} \right\|_{L^1(Q)}  \le C.
\end{equation}
From this, we obtain the counterpart of (\ref{eq:42}):
\begin{equation}
\label{eq:77}
\int_Q  (m_R(t,x)+ 1) \left |\frac{\partial H_R[m_R(t,\cdot)]}{\partial p}(x, \nabla u_R(t,x) )\right|^2dxdt
 \le C,
\end{equation}
where $C$ is a constant independent of $R$.
This estimate allows one for carrying out Steps 3 and 4 in \S~\ref{sec:priori-estimates-1} and obtaining estimates independent of $R$: in particular, the same Bernstein argument can be used, and   we obtain that there exists a constant independent of $R$ such that
$\|\nabla u_R\|_{L^\infty(Q)}\le C$. In turn, step 5 in \S~\ref{sec:priori-estimates-1} can be used and leads to estimates independent of $R$.
\\
From this, taking $R$ large enough yields the desired existence result.
\end{proof}

\section{Uniqueness}
\label{sec:uniqueness}

\subsection{Uniqueness for (\ref{eq:6})-(\ref{eq:8}): a sufficient condition}
\label{sec:suff-cond-uniq}
In what follows, we prove sufficient conditions leading to the uniqueness of a classical solution of (\ref{eq:6})-(\ref{eq:8}). For simplicity,
we still assume that the final cost does not depend on the density, i.e. that there exists a smooth function $u_T$ such that $h[m](x)=u_T(x)$.
In order to simplify the discussion, we assume that the operator $H$ depends smoothly enough  on its argument
to give sense to the calculations that follow.\\
We consider two classical solutions $(u,m)$ and $(\tilde u, \tilde m)$ of
\begin{eqnarray}
\label{eq:46}
0&=&
\begin{array}[t]{l}
\ds
\frac{\partial u} {\partial t} (t,x) + \nu \Delta u(t,x) + H[m(t,\cdot)]
( x, \grad u(t,x)) \\ \ds +
\int_{\T^d} \frac{\partial H} {\partial m}  [m(t,\cdot)]
(\xi, \grad u(t, \xi))(x) m(t,\xi) d\xi  ,
\end{array}
\\
\label{eq:47}
	0&=&\ds \frac{\partial m} {\partial t} (t,x)  - \nu \Delta  m(t,x) +
 \div\Bigl( m(t,\cdot) \frac{\partial H} {\partial p} [m(t,\cdot)](\cdot,\grad u(t,\cdot))\Bigr)(x) ,
\end{eqnarray}
and
\begin{eqnarray}
\label{eq:48}
0&=&
\begin{array}[t]{l}
\ds
\frac{\partial \tilde u} {\partial t} (t,x) + \nu \Delta \tilde u(t,x) + H[\tilde m(t,\cdot)]
( x, \grad \tilde u(t,x)) \\ \ds +
\int_{\T^d} \frac{\partial H} {\partial \tilde m}  [\tilde m(t,\cdot)]
(\xi, \grad \tilde u(t, \xi))(x) \tilde m(t,\xi) d\xi  ,
\end{array}
\\
\label{eq:49}
	0&=&\ds \frac{\partial \tilde m} {\partial t} (t,x)  - \nu \Delta  \tilde m(t,x) +
 \div\Bigl( \tilde m(t,\cdot) \frac{\partial H} {\partial p} [\tilde m(t,\cdot)](\cdot,\grad \tilde u(t,\cdot))\Bigr)(x) .
\end{eqnarray}
We subtract (\ref{eq:48}) from (\ref{eq:46}), multiply the resulting equation by $(m(t,x)-\tilde m(t,x))$, and integrate over $Q$.
Similarly, we subtract (\ref{eq:49}) from (\ref{eq:47}), multiply the resulting equation by $(u(t,x)-\tilde u(t,x))$, and integrate over $Q$.
We sum the two resulting identities:
 we obtain
\begin{equation}
  \label{eq:50}
  \begin{split}
0 =& \int_{T^d }   (u(T,x)-\tilde u(T,x))(m(T,x)-\tilde m(T,x)) dx \\ & -  \int_{T^d }(u(0,x)-\tilde u(0,x))(m(0,x)-\tilde m(0,x))\Bigr) dx \\
& + \int_{t=0}^T E [ m(t,\cdot),\nabla u(t,\cdot), \tilde m(t,\cdot), \nabla \tilde u(t,\cdot)] dt .
  \end{split}
\end{equation}
where
\begin{displaymath}
  \begin{split}
  & E[m_1,p_1,m_2, p_2] = \int_{\T^d}  (H[m_1]( x, p_1(x)) - H[m_2](x, p_2(x)) ) (m_1(x)-m_2(x)) dx \\
  & + \int_{\T^d}  (m_1(x)-m_2(x)) \int_{\T^d}    \left( \frac {\partial H}{\partial m}
 [m_1](\xi, p_1(\xi) )(x)m_1(\xi) -  \frac {\partial H}{\partial m}
 [m_2](\xi, p_2(\xi) )(x)m_2(\xi)\right) d\xi dx\\
& - \int_{\T^d}   \left(
 m_1(x) \frac {\partial}{\partial p} H[m_1]( x, p_1(x)) - m_2(x) \frac {\partial}{\partial p} H[m_2]( x, p_2(x))\right) (p_1(x)-p_2(x)) dx.
\end{split}
\end{displaymath}
Call $\delta m= m_2-m_1$ and $\delta p= p_2-p_1$   and consider the function $  e : [0,1]\to \R$ defined by
\begin{equation}
  \label{eq:51}
  \begin{array}[c]{rcll}
\ds e(\theta)&=&\ds \frac 1 \theta E[m_1,p_1, m_1+\theta \delta m, p_1 +\theta \delta p], \qquad        & \theta>0,\\
e(0)&=&0 .
  \end{array}
\end{equation}
It can be checked that $e$ is $\cC^1$ on $[0,1]$ and that its derivative is
 \begin{equation}
   \label{eq:52}
   \begin{split}
     &e'(\theta)= 2 \int_{\T^d}\int_{\T^d} \frac {\partial H}{\partial m}  [m_1 +\theta \delta m](\xi, p_1+ \theta \delta p (\xi) )(x)   \delta m(\xi) \delta m(x)\\
&+  \int_{\T^d}\int_{\T^d} \int_{\T^d}  (m_1 (\xi)  +\theta \delta m (\xi) )\frac {\partial^2 H}{\partial m \partial m}  [m_1 +\theta \delta m](\xi, p_1+\theta  \delta p (\xi) )(x)(y) \delta m(x) \delta m(y)\\
& -  \int_{\T^d}  (m_1 (x)  +\theta \delta m (x) )   \delta p (x) \cdot  D^2_{p,p} H  [m_1 +\theta \delta m] (x, p_1(x)+\theta \delta p (x)) \delta p(x).
   \end{split}
 \end{equation}
Let us introduce the functional defined on $\cC(\T^d)\times \cC(\T^d; \R^d)$ by
\begin{equation}
  \label{eq:53}
\cH [m,p]:=\int_{\T^d}  m (x)  H  [ m] (x, p(x)) dx.
\end{equation}
The second order Fr{\'e}chet derivative  of $\cH$  with respect to $m$  (respectively $p$) at $(m,p)$ is a bilinear form on $\cC(\T^d)$,
(resp. $\cC(\T^d;\R^d)$), noted $D^2_{m,m} \cH[m,p] $, (resp.    $D^2_{p,p} \cH [m,p]$) .
For all $m\in \cC(\T^d) \cap \PP$ and all $  p\in   \cC (\T^d;\R^d) $, let us define the quadratic form
$\cQ[m,p]$  on $\cC(\T^d)\times \cC(\T^d;\R^d)$  by
\begin{equation}
  \label{eq:54}
 \cQ[m,p] (\mu, \pi)=  D^2_{m,m} \cH [m, p](\mu, \mu)
- D^2_{p,p} \cH [m,p](\pi, \pi).
\end{equation}
We see that (\ref{eq:52}) can be written as follows:
\begin{equation}\label{eq:55}
  e '(\theta)= \cQ[m_1 +\theta \delta m,p_1 +\theta \delta p](\delta m, \delta p).
\end{equation}
\begin{theorem}
  \label{sec:uniqueness-1}
We assume $(H_0)$ and that $(m,x,p)\mapsto H[m](x,p)$ is $\cC^2$ on $\cC(\T^d)\times \T^d\times \R^d$.
A sufficient condition for the uniqueness of a classical solution of (\ref{eq:6})-(\ref{eq:8})
is that
\begin{enumerate}
\item for all $m\in \cC(\T^d) \cap \PP$ and all $  p\in   \cC(\T^d;\R^d) $, the quadratic form \\
$\mu\mapsto D^2_{m,m} \cH [m, p] (\mu,\mu)$ is  positive definite
\item for all $m\in \cC(\T^d) \cap \PP$ and $x\in \T^d$, the real valued function
 $p\in \R^d \mapsto H[m](x,p)$ is strictly concave.
\end{enumerate}
\end{theorem}
\begin{proof}
From  the concavity of  $p\mapsto H[m](x,p)$,  $-D^2_{p,p}\cH[m,p]$ is positive semi-definite.
Therefore, $\cQ[m,p]$ is  positive semi-definite, and $\cQ[m,p](\mu,\pi)=0$ implies that
$ D^2_{m,m} \cH [m, p](\mu, \mu)=0$ and  $
- D^2_{p,p} \cH [m,p](\pi, \pi)=0$, and therefore $\mu=0$.
\\
  From (\ref{eq:50}), two solutions $(u,m)$ and $(\tilde u, \tilde m)$ of  (\ref{eq:6})-(\ref{eq:8}) satisfy
  \begin{equation}
    \label{eq:56}
    \int_{t=0}^T E [ m(t,\cdot),\nabla u(t,\cdot), \tilde m(t,\cdot), \nabla \tilde u(t,\cdot)] dt=0,
  \end{equation}
because $\tilde m(0,\cdot)=m(0,\cdot)$ and $\tilde u(T,\cdot)=u(T,\cdot)$.\\
But, from (\ref{eq:51}) and (\ref{eq:55}),
the properties of the quadratic form $\cQ[  (1-\theta) m(t,\cdot) +\theta \tilde m(t,\cdot),    (1-\theta) \nabla u(t,\cdot) +
 \theta \nabla \tilde u (t,\cdot)]$ imply that \\
 $ \int_{t=0}^T E [ m(t,\cdot),\nabla u(t,\cdot), \tilde m(t,\cdot), \nabla \tilde u(t,\cdot)] dt >0$
if  $m\not= \tilde m$.
\\
Therefore, (\ref{eq:56}) implies that $m=\tilde m$.
Then,
\begin{equation}\label{eq:57}
\begin{split}
  &0= \\ & \int_{Q}  m(t,x)  \left(
 \frac {\partial H}{\partial p} [m(t,\cdot)]( x, \nabla u(t,x)) - \frac {\partial H}{\partial p}[m(t,\cdot)]( x, \nabla \tilde u(t,x))
\right)\cdot (\nabla  u-\nabla \tilde u) (t,x).
      \end{split}
\end{equation}
  If  $\nu>0$, then the maximum principle implies that $m(t,x)>0$ for all $t>0,x\in \T^d$. This observation, (\ref{eq:57}) and
  the strict concavity of $H$ with respect to $p$ imply that $\nabla u(t,x)= \nabla \tilde u (t,x)>0$ for all $t,x$,
which yields immediately that $u=\tilde u$ by using (\ref{eq:6}).
\end{proof}

\begin{remark}\label{sec:uniq-refeq:6-refeq:8-1}
Let us give an alternative argument which does not require the knowledge that  $m(t,x)>0$ for all $t>0,x\in \T^d$.
Such an argument may be useful in situations when $\nu=0$ or  $\nu$  is replaced in (\ref{eq:1})  by a function of $x$
 which vanishes in some regions of $\T ^d$.  The strict concavity of $H$ with respect to $p$ and  (\ref{eq:57})
yield the fact that $u=\tilde u$ in the region where $m>0$. This implies that
$G[m (t,\cdot)](x,\nabla u (t,x) )=G[m (t,\cdot)](x,\nabla \tilde u (t,x) )$: hence, for all $t$ and $x$,
\begin{displaymath}
  \begin{split}
&\frac{\partial  u} {\partial t} (t,x) + \nu \Delta u(t,x) + H[m(t,\cdot)]
( x, \grad u(t,x))  \\=&\frac{\partial \tilde u} {\partial t} (t,x) + \nu \Delta \tilde u(t,x) + H[m(t,\cdot)]
( x, \grad \tilde u(t,x))  .
  \end{split}
\end{displaymath}
We can then apply standard results on the uniqueness of the Cauchy problem with the HJB equation
$\frac{\partial  u} {\partial t} (t,x) + \nu \Delta u(t,x) + H[m(t,\cdot)]
( x, \grad u(t,x)) = g$ and obtain that $u=\tilde u$.
\end{remark}

\begin{corollary}
  \label{sec:suff-cond-uniq-2}
In the  case when  $H$ depends
 locally on $m$, i.e. \[H[m](x,p)= \tilde H(x, p, m(x)),\]
the sufficient condition in Theorem \ref{sec:uniqueness-1} is implied by
the strict concavity of $p\in \R^d \mapsto  \tilde H (x, p,m)$ for all $m>0$ and $x\in \T^d$ and
the strict convexity of the
real valued function  $m\in \R_+\mapsto  m \tilde H (x, p,m)$, for all $p\in \R^d$.
\end{corollary}

\paragraph{Example}
Consider for example the Hamiltonian
\begin{equation}
  \label{eq:58}
 H[m](x,p)= \tilde H(x, p, m(x))= -\frac { |p|^\beta}{ (c+  m(x) )^\alpha} + F(m(x)),
\end{equation}
with $c>0$, $\alpha>0$, $\beta > 1$,  $F$ a smooth function defined on $\R_+$.
One can check that if $\alpha\le 1$ and $F$ is strictly convex, then uniqueness holds.\\
Such a Hamiltonian arise in a local model for  congestion, see \cite{PLL}.

\begin{remark}
  \label{sec:uniq-refeq:6-refeq:8}
The same analysis can be carried out for mean field games, see \cite{PLL}:  for example, under Assumption $(H_0)$ and
in the  case when  $H$ depends  locally on $m$, i.e. $H[m](x,p)= \tilde H(x, p, m(x))$,
a sufficient condition for the uniqueness of a classical solution of (\ref{eq:9})-(\ref{eq:11}) is that
\begin{displaymath}
  \begin{pmatrix}
 2 \frac {\partial \tilde H} {\partial m}    \left(x,p,m\right) & -   \frac {\partial } {\partial m}  \nabla_p ^T \tilde H(x,p,m)
\\
-   \frac {\partial } {\partial m}  \nabla_p \tilde H(x,p,m) & - 2D^2_{p,p} \tilde H(x,p,m)
 \end{pmatrix}
\end{displaymath}
be positive definite for all $x\in \T^d$, $m>0$ and $p\in \R^d$.
Here, we see that the sufficient condition involves the mixed partial derivatives of $\tilde H$ with respect to $m$ and $p$, which is not the case
for mean field type control. If $\tilde H$ depends separately on $p$ and $m$ as in \cite{MR2295621}, then
 $\frac {\partial } {\partial m}  \nabla_p \tilde H (x,p,m)=0$ and  the condition becomes:
$\tilde H$ is strictly concave with respect to $p$ for $m>0$ and non decreasing with respect to $m$, (or concave with respect to $p$ and strictly increasing with respect to $m$).
\end{remark}

\begin{remark}
  \label{sec:uniqueness_of_weak_solutions}
The extension of the result  on uniqueness to weak solutions is not trivial. In the context of mean filed games, one can find such results in \cite{MR3090129} and \cite{porretta2015}: roughly speaking they rely on some new uniqueness results for weak solutions of the Fokker-Planck equation and on crossed regularity lemmas, see Lemma 5 in \cite{MR3090129}.
In the context of mean field type control, the same kind of analysis has not been done yet.
\end{remark}

In the case when $n=d$, $g[m](x,v)=v$ and $ v\mapsto f[m](x,v)$ is strictly convex for all $m\in \PP$ and
$x\in \T^d$, it is well known that
$
f[m](x,v)= \sup_{q\in \R^d} \left(  H[m] (x, q) - q\cdot v \right)$.
Furthermore if $p\mapsto H[m](x,p)$ is strictly concave for all $m\in \PP$ and
$x\in \T^d$, then
\begin{equation}
 \label{eq:59}
f[m](x,v)= \max_{q\in \R^d} \left(  H[m] (x, q) - q\cdot v \right)
\end{equation}
 and the maximum is achieved by a unique
$q$. This observation leads to the following necessary condition for  the assumption of Theorem \ref{sec:uniqueness-1}
to be satisfied.
\begin{proposition}
  \label{sec:suff-cond-uniq-4}
Assume that  $n=d$, $g[m](x,v)=v$, that $ v\mapsto f[m](x,v)$ is strictly convex for all $m\in \PP$ and $x\in \T^d$,
 and that $p\mapsto H[m](x,p)$ is strictly concave for all $m\in \PP$ and $x\in \T^d$.
If for all $p \in \cC(\T ^d;\R^d)$,  $m\mapsto \cH[m,p]$ is strictly convex in  $\PP\cap \cC(\T^d)$,
then for all $v \in \cC(\T ^d;\R^d)$,
$m\mapsto \int_{\T^d} m(x) f[m](x,v(x)) dx $ is strictly convex in   $\PP\cap \cC(\T^d)$.
\end{proposition}
\begin{proof}
  Take $\lambda_1> 0$ and $\lambda_2> 0$ such that $\lambda_1+\lambda_2=1$ and $m_1\not=m_2$ in
$ \PP\cap \cC(\T^d)$. From (\ref{eq:59}),
  \begin{displaymath}
    \begin{split}
      &      \int_{\T^d} (\lambda_1 m_1(x)+\lambda_2 m_2(x))  f[\lambda_1 m_1+\lambda_2 m_2](x,v(x)) dx\\
      =& \int_{\T^d} \max_{q\in \R ^d }  (\lambda_1 m_1(x)+\lambda_2 m_2(x))
 \left(  H[\lambda_1 m_1+\lambda_2 m_2] (x, q) - qv(x) \right) dx.
\end{split}
\end{displaymath}
If for all $x\in \T^d$, the maximum in the latter integrand  is achieved by $q^*(x)$,
 then $x\mapsto q^*(x)$ is a continuous function (from the continuity of $v$) and we have
 \begin{displaymath}
    \begin{split}
      &      \int_{\T^d} (\lambda_1 m_1(x)+\lambda_2 m_2(x))  f[\lambda_1 m_1+\lambda_2 m_2](x,v(x)) dx\\
 =&   \max_{q\in \cC(T^d;\R ^d) }
\int_{\T^d}  (\lambda_1 m_1(x)+\lambda_2 m_2(x))
 \left(  H[\lambda_1 m_1+\lambda_2 m_2] (x, q(x)) - q(x)v(x) \right) dx.
    \end{split}
  \end{displaymath}
From this and  the convexity of $m\mapsto \cH[m,p]$, we deduce that
 \begin{displaymath}
    \begin{split}
      &      \int_{\T^d} (\lambda_1 m_1(x)+\lambda_2 m_2(x))  f[\lambda_1 m_1+\lambda_2 m_2](x,v(x)) dx\\
< &  \max_{q\in \cC(T^d;\R ^d) }   \left(
  \begin{array}[c]{l}
  \lambda_1  \int_{\T^d}   \left(m_1(x)  H[ m_1] (x, q(x)) - q(x)v(x) \right) dx
+\\
  \lambda_2  \int_{\T^d}   \left(m_2(x)  H[ m_2] (x, q(x)) - q(x)v(x) \right) dx
  \end{array}
\right)
\\
\le & \lambda_1 \max_{q\in \cC(T^d;\R ^d) }   \int_{\T^d}   \left(m_1(x)  H[ m_1] (x, q(x)) - q(x)v(x) \right) dx \\
&+
  \lambda_2 \max_{q\in \cC(T^d;\R ^d) }  \int_{\T^d}   \left(m_2(x)  H[ m_2] (x, q(x)) - q(x)v(x) \right) dx
\\
= & \lambda_1 \int_{\T^d}  m_1(x) f[m_1](x,v(x)) dx + \lambda_2  \int_{\T^d}  m_2(x) f[m_2](x,v(x)) dx .
\end{split}
\end{displaymath}
\end{proof}

\subsection{Back to the control of McKean-Vlasov dynamics}
\label{sec:back-control-mckean}
As in the end of the previous paragraph, we assume that  $n=d$ and $g[m](x,v)=v$. The control of McKean-Vlasov dynamics
can be written as a control problem with linear constraints by making the change of variables $z=m v$: it consists of minimizing
\begin{equation}
  \label{eq:60}
  \tilde J(z,m_z)=\int_{Q} f[m_z(t,\cdot)]\left ( x, \frac {z(t,x)}{m_z(t,x)}\right) m_z(t,x) dx dt
+ \int_{\T^d} u_T(x) m_z(T,x) dx,
\end{equation}
subject to the linear constraints
 \begin{equation}
\label{eq:61}
\frac{\partial m_z} {\partial t} (t,x)  - \nu \Delta m_z(t,x) + \div z(t,x)=0 \;\; t\in (0,T], x\in \T^d,
 \end{equation}
with the initial condition
\begin{equation}
\label{eq:62}
m_z(0,x)= m_0(x),\quad x\in \T^d.
\end{equation}
For simplicity, we assume that $f$ depends locally on $m$, i.e.
\begin{equation*}
	 f[m]\left(x,v(x)\right)  =   \tilde f\left(x,v(x),m(x)\right), \quad \forall x \in \T^d.
\end{equation*}
We are going to look for
sufficient conditions for $(z,m)\mapsto   m \tilde f (x,\frac z m,m)$ be a convex function.
This condition will thus yield the uniqueness for the above control problem.\\
Assuming that all the following differentiations are  allowed, we see that the Hessian of the
latter function is
\begin{align*}
&\Theta(x,v,m)=
\begin{pmatrix}
 \frac{1}{m^3} z\cdot D^2_{vv}  \tilde f\left(\frac{z}{m},m\right) z &
 -\frac{1}{m^2}   z \cdot D^2_{vv}  \tilde f\left(\frac{z}{m},m\right)  \\
  -\frac{1}{m^2} D^2_{vv}  \tilde f\left(\frac{z}{m},m\right)  z & \frac{1}{m} D^2_{vv} \tilde f\left(\frac{z}{m},m\right) \\
 \end{pmatrix} \\
+
&\begin{pmatrix}
 2 \frac{\partial  \tilde f}{\partial m}   \left( \frac{z}{m},m\right) +  m \frac{\partial^2 \tilde f}{\partial m^2}  \left(\frac{z}{m},m\right)
- 2  \frac{z}{m} \cdot \frac{\partial \nabla_v  \tilde f}{\partial m} \left(\frac{z}{m},m\right)
  & \frac{\partial \nabla_v^T  \tilde f}{\partial m} \left(\frac{z}{m},m\right) \\
  \frac{\partial  \nabla_v \tilde f  }{\partial m}  \left (\frac{z}{m},m\right) & 0
 \end{pmatrix}
\end{align*}
where we have omitted the dependency on $x$ for brevity. This
 is better understood when expressed in terms of $(v,m)$:
\begin{equation}\label{eq:63}
\Theta(x,v,m)=\begin{pmatrix}
 \frac{\partial^2}{\partial m^2}  \left( m  \tilde f\left(x,v,m\right) \right) &
m\frac{\partial  \nabla_v  ^T \tilde f}{\partial m} \left(x,v,m\right) \\
  m\frac{\partial  \nabla _v \tilde f}{\partial m} \left(x,v,m\right) & m D^2_{vv}  \tilde f\left(x,v,m\right)
 \end{pmatrix}	.
\end{equation}
We have proved the following
\begin{proposition}
  \label{sec:back-control-mckean-4}
We assume that  $n=d$ and $g[m](x,v)=v$, and that
$ f[m]\left(x,v(x)\right)  =   \tilde f\left(x,v(x),m(x)\right)$,  for all $x \in \T^d$,
where  $\tilde f$ is a smooth function.
A sufficient condition for the uniqueness of a minimum $( z^*,m^*)$ such that $m^*>0$
is that $\Theta(x,v,m)$ be positive definite for all $x\in \T^d$, $m>0$ and $v\in \R^d$.
\end{proposition}
\begin{proposition}\label{sec:back-control-mckean-3}
We make  the same assumptions as in Proposition~\ref{sec:back-control-mckean-4}.
The positive definiteness of  $\Theta(x,v,m)$  for all $x\in \T^d$, $m>0$ and $v\in \R^d$
implies  the sufficient conditions on $\tilde H$ in Corollary \ref{sec:suff-cond-uniq-2}.
\end{proposition}
\begin{proof}
We observe first that the positive definiteness of $\Theta$ implies that $D_{vv}^2 \tilde f (x,v,m)$
is positive definite for all $x\in \T^d$, $m>0$ and $v\in \R^d$.
\\
Let us call $v^*\in \R^d$ the vector achieving
$ \tilde H(x,p,m)=p\cdot v^* +\tilde f(x,v^*,m)$. We know that $\nabla_p \tilde H(x,p,m)= v^*$.
 Differentiating the optimality condition for $v^*$
 with respect to $p$, we find that
 \begin{equation}\label{eq:64}
D_{p,p}^2 \tilde H(x,p,m)=- \left(D^2_{v,v} \tilde f(x,v^*,m)\right)^{-1}.
 \end{equation}
Note that (\ref{eq:64}) implies the strict concavity of
$p\mapsto \tilde H(x,p,m)$ which is the first desired condition on $\tilde H$.
 The second condition on $\tilde H$ will be a consequence of the implicit function theorem:
differentiating $\tilde H$ with respect to $m$, we find that
\begin{equation}\label{eq:65}
   \frac{\partial \tilde H}{\partial m}(x,p,m)=  \frac{\partial \tilde f}{\partial m} (x,v^*,m) +
 \nabla_v \tilde f (x,v^*,m) \cdot  \frac{\partial v^*}{\partial m}+ p  \cdot  \frac{\partial v^*}{\partial m} =
  \frac{\partial \tilde f}{\partial m} (x,v^*,m) ,
\end{equation}
 where the last identity comes from the definition of $v^*$. Differentiating once more with respect to $m$, we find that
 \begin{equation}\label{eq:66}
   \frac{\partial^2 \tilde H}{\partial m^2}(x,p,m)=  \frac{\partial^2 \tilde f}{\partial m^2}(x,v^*,m) +
 \frac{\partial \nabla_v \tilde f }{\partial m}(x,v^*,m) \cdot  \frac{\partial v^*}{\partial m}.
 \end{equation}
Then the implicit function theorem applied to the optimality condition for $v^*$ yields that
\begin{equation}\label{eq:67}
  \frac{\partial v^*}{\partial m}= - \left( D^2 _{v,v}\tilde  f (x,v^*,m)\right)^{-1} \frac{\partial \nabla_v \tilde f }{\partial m}(x,v^*,m).
\end{equation}
From (\ref{eq:64})-~(\ref{eq:67}), we see that
\begin{equation}
  \label{eq:68}
  \begin{split}
 &\frac{\partial^2  }{\partial m^2}   \left(m\tilde H(x,p,m)\right)  =
\frac{\partial^2}{\partial m^2}  \left( m  \tilde f\left(x,\cdot,m \right) \right)(v^*) \\ &
- \left (m  \frac{\partial \nabla_v \tilde f }{\partial m} (x,v^*,m) \right)\cdot
\left(m  D^2 _{v,v}\tilde  f (x,v^*,m) \right)^{-1} \left (m  \frac{\partial \nabla_v \tilde f }{\partial m} (x,v^*,m) \right)   .
  \end{split}
\end{equation}
Hence,  $\frac{\partial^2  }{\partial m^2}   \left(m\tilde H(x,p,m)\right)$
is a Schur complement of $\Theta(x,v^*, m)$. Therefore, it is positive definite and we have proved the second condition on $\tilde H$.\\
%All the arguments aboved can be reversed by Fenchel duality. The conditions on $\Theta$ and $\tilde H$ are equivalent.
\end{proof}

\section{Numerical Simulations}
\label{sec:some-simulations}
Here we model a situation in which a crowd of pedestrians   is driven to leave a given square hall (whose side is 50 meters long) containing rectangular  obstacles:
 one can imagine for example a situation of panic in
a closed building, in which the population tries to reach the exit doors. The chosen geometry is represented on Figure~\ref{fig:1}.
\begin{figure}[htbp]
  \centering
  \includegraphics[width=3cm]{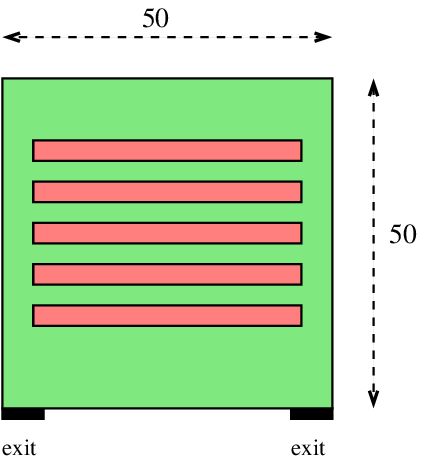}  \includegraphics[width=5cm]{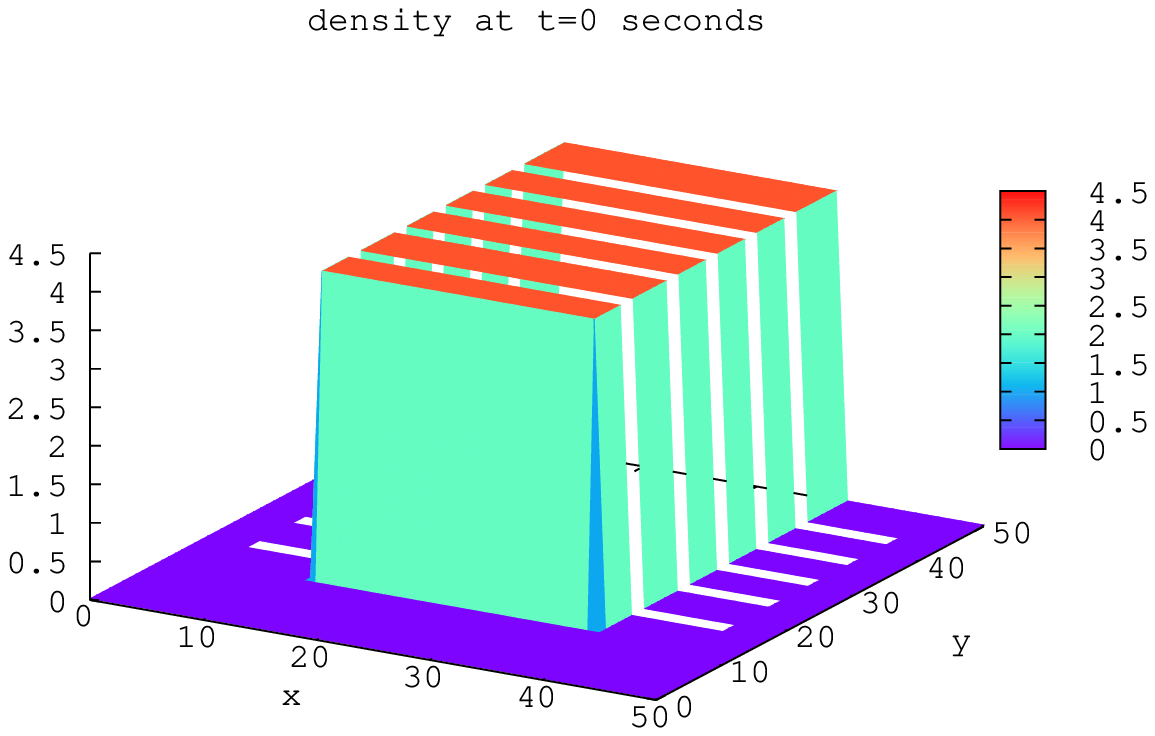}
  \caption{Left: the geometry. Right: the density at $t=0$}
  \label{fig:1}
\end{figure}
The aim is to compare the evolution of the density in two models:
\begin{enumerate}
\item Mean field games: we choose $\nu=0.012$ and  the Hamiltonian to be of the form (\ref{eq:58}), i.e. which takes congestion effects
 into account  and depends locally  on $m$; more precisely:
  \begin{displaymath}
    \tilde H(x,p,m)= -\frac {8 |p|^2} {(1+m)^{\frac 3 4}} + \frac 1 {3200}.
  \end{displaymath}
The system (\ref{eq:9})-~(\ref{eq:10}) becomes
\begin{eqnarray}
  \label{eq:69}
   \frac {\partial  u}{\partial t}  +   0.012 \;  \Delta  u
     -\frac {8} {(1+m)^{\frac 3 4}} \;   |\nabla u|^2&=&  -    \frac  1  {3200},
\\\label{eq:70}
  \frac {\partial  m}{\partial t}  -   0.012\;  \Delta m   - 16\div\left(   \frac { m \nabla u } {(1+m)^{\frac 3 4}}  \right)&=&0.
\end{eqnarray}
The horizon $T$  is $T=50$ minutes. There is no terminal cost. \\
There are two exit doors, see Figure~\ref{fig:1}. The part of the boundary corresponding to the doors is called  $\Gamma_D$. The boundary conditions at the exit doors are chosen as follows:  there is a Dirichlet condition for $u$ on $\Gamma_D$, corresponding to an exit cost; in our simulations, we have chosen $u=0$ on $\Gamma_D$.
For $m$, we may assume that $m =0$ outside the domain, so we also get the Dirichlet condition  $m=0$ on $\Gamma_D$.\\
The boundary $\Gamma_N$  corresponds to the solid walls of the hall and of the obstacles.
A natural boundary condition for $u$ on $\Gamma_N$ is a homogeneous Neumann boundary condition, i.e.
$ \frac{\partial u}{\partial n}=0$ which says that the velocity of the pedestrians is tangential to the walls.
 The natural condition for the density $m$ is that
$\nu \frac {\partial  m}{\partial n}+m     \frac {\partial \tilde H} {\partial p} (\cdot, \nabla  u,m )\cdot n=0$, therefore $ \frac {\partial  m}{\partial n}=0$
on $\Gamma_N$ .
\item Mean field type control:  this is the situation where pedestrians or robots use the same feedback law (we may imagine that they follow the strategy  decided by a leader); we keep the same Hamiltonian, and the HJB equation becomes
  \begin{equation}
    \label{eq:71}
   \frac {\partial  u}{\partial t}  +   0.012 \;  \Delta  u   -
     \left( \frac {2} {(1+m)^{\frac 3 4}}    +   \frac {6} {(1+m)^{\frac 7 4}}\right) \;   |\nabla u|^2= -     \frac  1  {3200}.
  \end{equation}
while (\ref{eq:70}) and the boundary condition are unchanged.
\end{enumerate}
The initial density $m_0$ is piecewise constant and takes two values $0$ and $4$ people/m$^2$, see Figure \ref{fig:1}.
At $t=0$, there are 3300 people in the hall.\\
We use the finite difference method originally proposed in \cite{MR2679575},
see \cite{MR3135339} for some details on the implementation and \cite{MR3097034} for convergence results. \\
On Figure \ref{fig:2}, we plot the density $m$ obtained by the simulations for the two models, at $t=1$, $2$, $5$ and $15$ minutes.
With both models, we see that the pedestrians rush towards the narrow corridors leading to the exits, at the left and right sides of the hall, and that the density reaches high values
at the intersections of corridors; then congestion effects explain why the velocity is low (the gradient of $u$)
in the regions where the density is high. On the figure, we see that the mean field type control leads to a slower exit of the hall,
 with lower peaks of density.
\begin{figure}[c]
  \begin{displaymath}
   \begin{array}[c]{cc}
  \!\!\!\!\!\!\!\!\! \includegraphics[width=5cm]{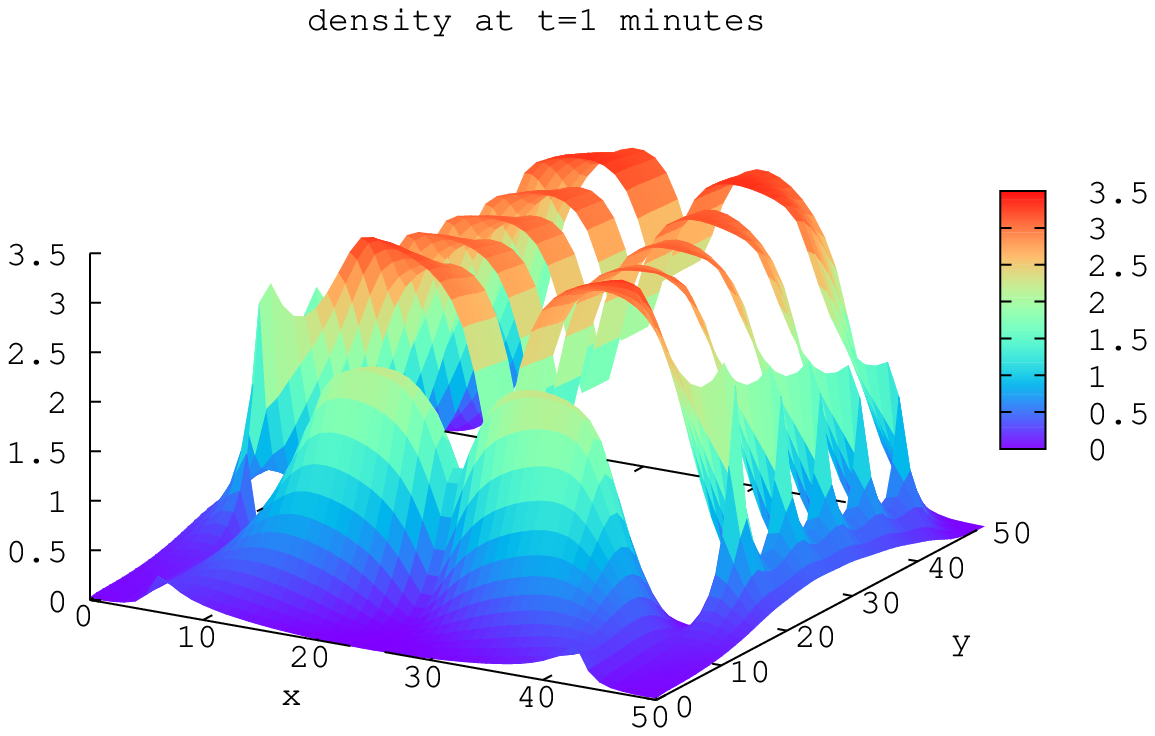}   \quad  & \includegraphics[width=5cm]{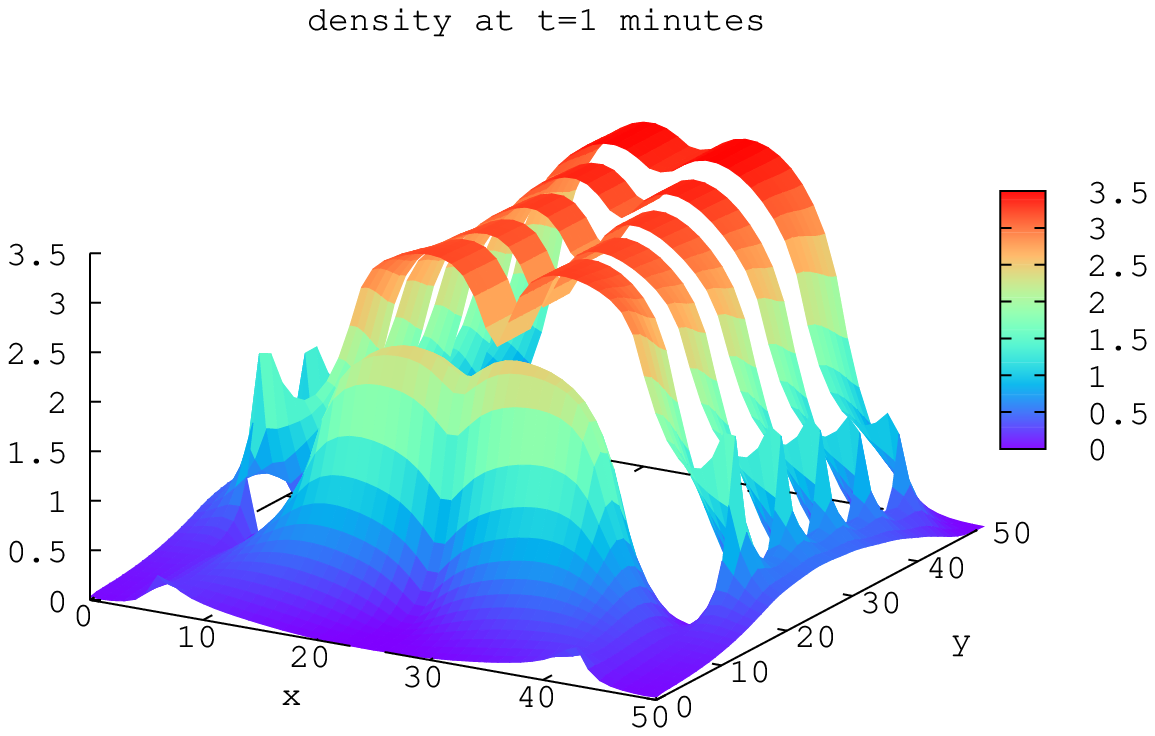} \\
  \!\!\!\!\!\!\!\!\! \includegraphics[width=5cm]{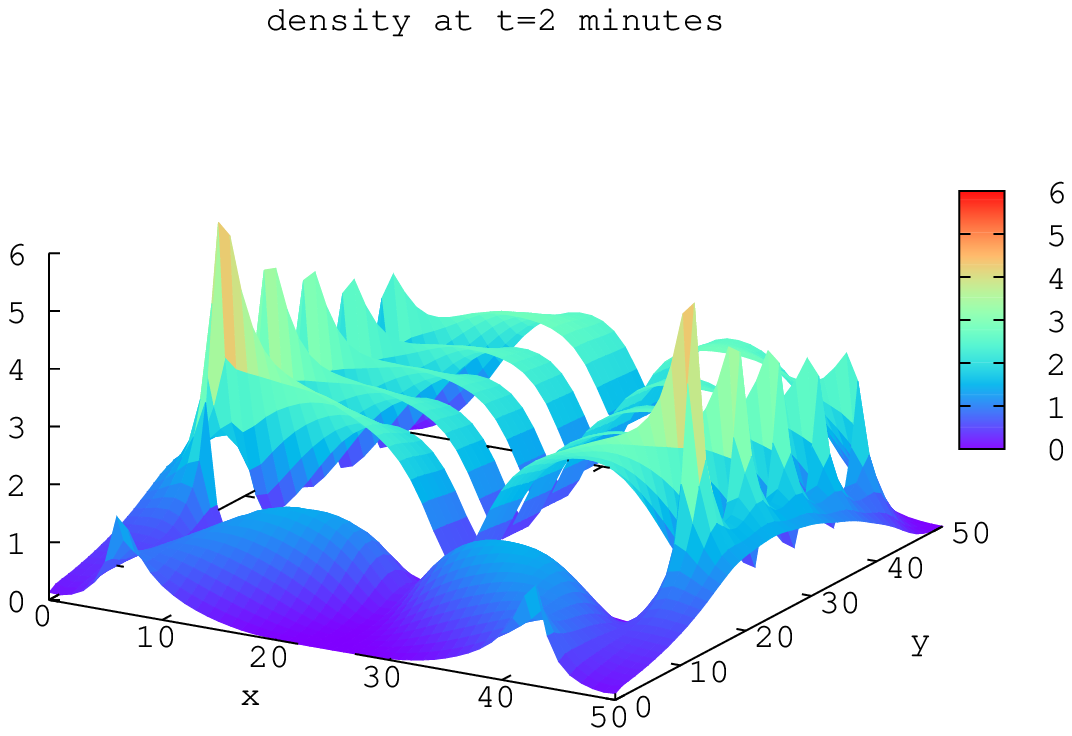}   \quad  & \includegraphics[width=5cm]{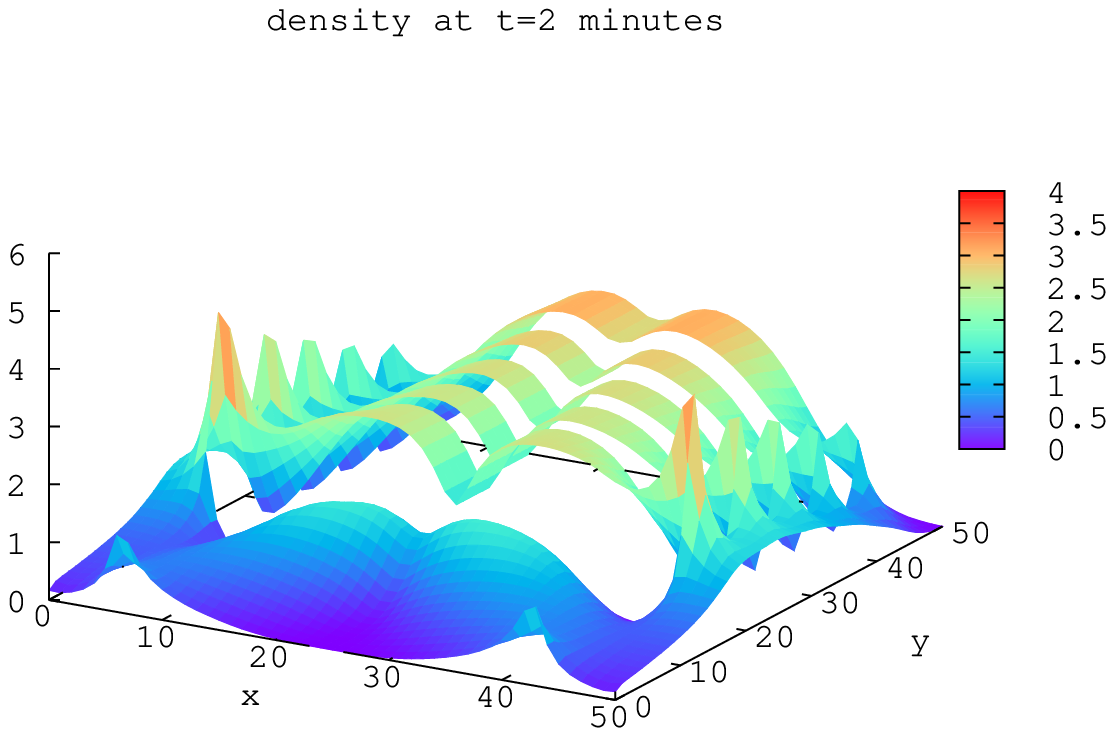} \\
  \!\!\!\!\!\!\!\!\! \includegraphics[width=5cm]{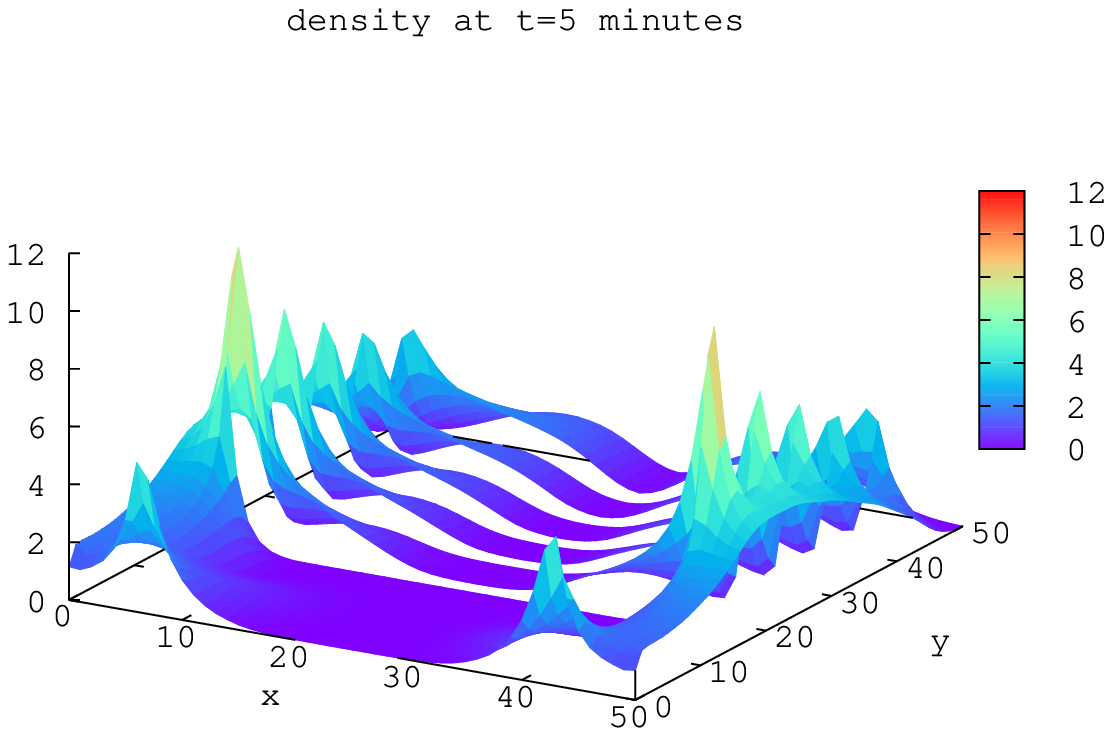}   \quad  &
\includegraphics[width=5cm]{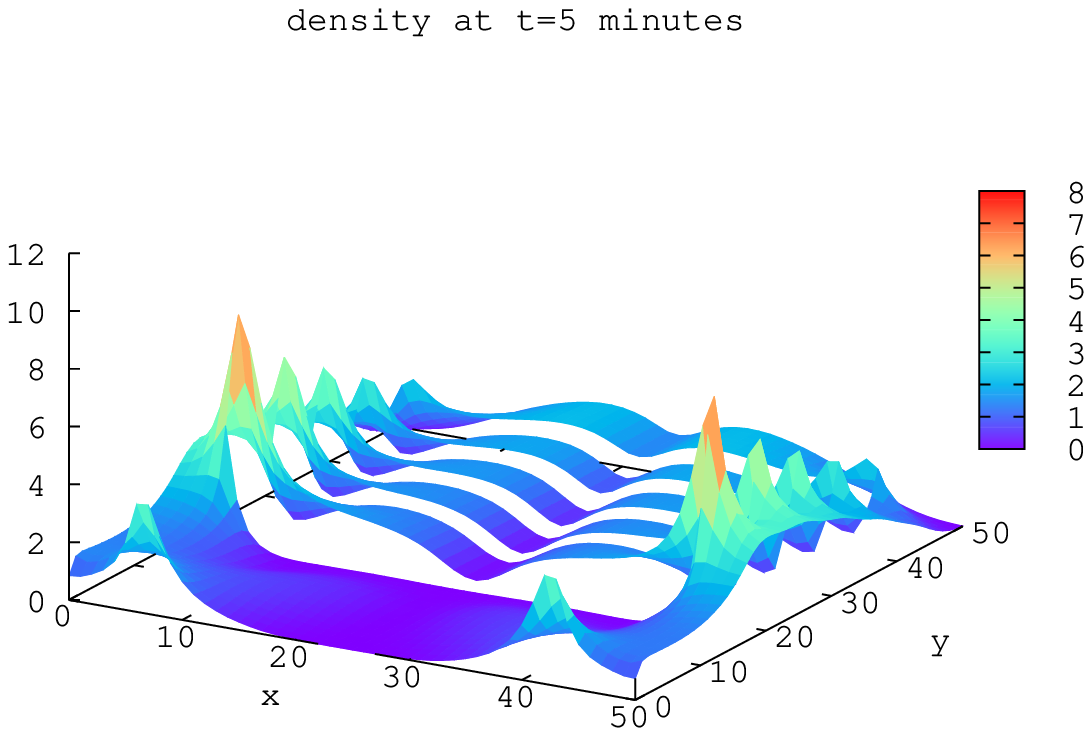} \\
  \!\!\!\!\!\!\!\!\! \includegraphics[width=5cm]{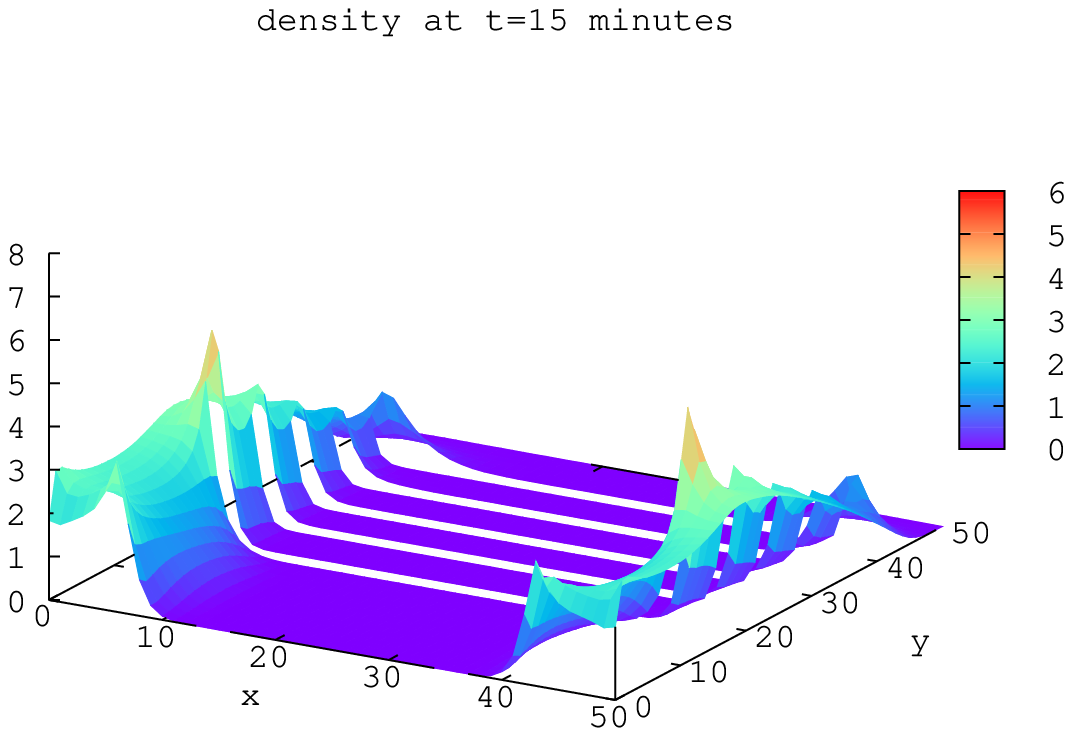}   \quad  & \includegraphics[width=5cm]{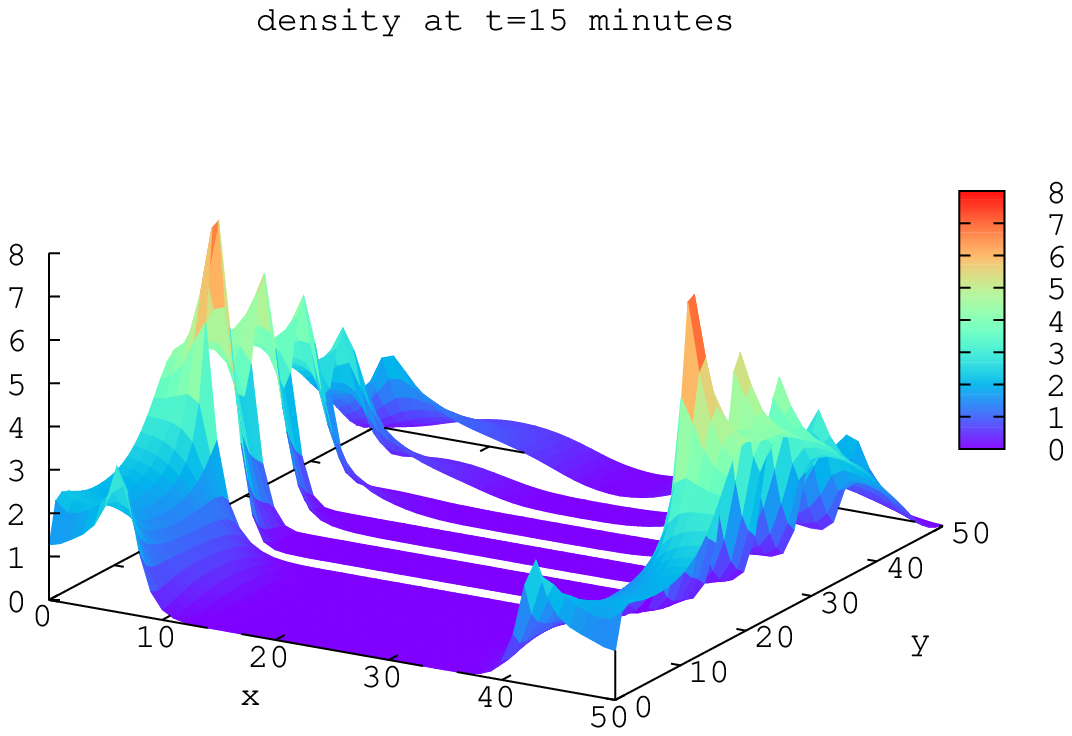} \end{array}
\end{displaymath}
  \caption{The density computed with the two models at different dates. Left: Mean field game. Right: Mean field type control.
The scales vary from one date to the other}
  \label{fig:2}
\end{figure}

\section*{\bf Acknowledgements}
We warmly thank A. Bensoussan for helpful  discussions.
 The first author  was partially funded  by the ANR projects ANR-12-MONU-0013 and ANR-12-BS01-0008-01.
The second author was partially funded by the Research Grants Council of HKSAR (CityU 500113).

\bibliographystyle{amsplain}
\bibliography{MFTC}

\end{document}